\documentclass[11pt]{amsart}

\usepackage{amsmath}
\usepackage{amssymb}
\usepackage{amscd}
\usepackage[english]{babel}
\usepackage{csquotes}

\usepackage{enumitem}
\usepackage{mathtools}
\usepackage[dvipsnames]{xcolor}

\usepackage{tikz}
\usepackage{tikz-cd}
\usepackage[mark=o]{dynkin-diagrams}
\usepackage{float}
\usepackage{ytableau}
\usepackage[doi=false, isbn=false, url=false]{biblatex}
\AtEveryBibitem{\clearlist{language}}
\usepackage[textsize=tiny]{todonotes}

\DeclarePairedDelimiter\abs{\lvert}{\rvert}

\newcommand{\kk}{{\mathsf{k}}}

\newcommand{\PP}{{\mathbb{P}}}

\newcommand{\RR}{{\mathbb{R}}}
\newcommand{\ZZ}{{\mathbb{Z}}}

\newcommand{\CA}{{\mathcal{A}}}
\newcommand{\CB}{{\mathcal{B}}}

\newcommand{\CE}{{\mathcal{E}}}
\newcommand{\CF}{{\mathcal{F}}}
\newcommand{\CG}{{\mathcal{G}}}
\newcommand{\CH}{{\mathcal{H}}}
\newcommand{\CK}{{\mathcal{K}}}
\newcommand{\CL}{{\mathcal{L}}}

\newcommand{\CO}{{\mathcal{O}}}
\newcommand{\CP}{{\mathcal{P}}}
\newcommand{\CQ}{{\mathcal{Q}}}
\newcommand{\CR}{{\mathcal{R}}}
\newcommand{\CS}{{\mathcal{S}}}
\newcommand{\CT}{{\mathcal{T}}}
\newcommand{\CU}{{\mathcal{U}}}
\newcommand{\CV}{{\mathcal{V}}}
\newcommand{\CW}{{\mathcal{W}}}

\newcommand{\Hom}{\mathop{\mathsf{Hom}}\nolimits}
\newcommand{\Ext}{\mathop{\mathsf{Ext}}\nolimits}

\newcommand{\Spec}{\mathop{\mathsf{Spec}}\nolimits}

\newcommand{\Ker}{\mathop{\mathsf{Ker}}\nolimits}

\newcommand{\rk}{\mathop{\mathsf{rk}}}

\newcommand{\Gr}{{\mathsf{Gr}}}
\newcommand{\OGr}{{\mathsf{OGr}}}
\newcommand{\LGr}{{\mathsf{LGr}}}
\newcommand{\IGr}{{\mathsf{IGr}}}
\newcommand{\IFl}{{\mathsf{IFl}}}

\newcommand{\Fl}{{\mathsf{Fl}}}

\newcommand{\GL}{{\mathsf{GL}}}

\newcommand{\SP}{{\mathsf{Sp}}}
\newcommand{\Spin}{{\mathsf{Spin}}}

\newcommand{\Coh}{{\mathop{\mathsf{Coh}}}}

\newcommand{\Rep}{{\mathsf{Rep}}}

\newcommand{\bfG}{\mathbf{G}}
\newcommand{\bfP}{\mathbf{P}}
\newcommand{\bfU}{\mathbf{U}}
\newcommand{\bfL}{\mathbf{L}}

\newcommand{\YD}{\mathrm{Y}}
\newcommand{\YDm}{\YD^{\mathrm{m}}}
\newcommand{\YDu}{\YD^{\mathrm{u}}}
\newcommand{\YDmu}{\YD^{\mathrm{mu}}}

\newcommand{\rmP}{\mathrm{P}}

\newcommand{\lsup}[2]{\vphantom{#1}^{#2}\!#1}
\newcommand{\lperp}[1]{\lsup{#1}{\perp}}

\newcommand{\ydw}{\mathsf{w}}
\newcommand{\ydh}{\mathsf{h}}

\theoremstyle{plain}

\newtheorem{theorem}{Theorem}[section]
\newtheorem{conjecture}[theorem]{Conjecture}
\newtheorem*{conjecture*}{Conjecture}
\newtheorem{lemma}[theorem]{Lemma}
\newtheorem{proposition}[theorem]{Proposition}
\newtheorem{corollary}[theorem]{Corollary}

\theoremstyle{definition}

\newtheorem{definition}[theorem]{Definition}

\theoremstyle{remark}

\newtheorem{remark}[theorem]{Remark}
\newtheorem{example}[theorem]{Example}

\title{Derived Categories of Grassmannians: a Survey}
\author{Anton Fonarev}
\dedicatory{}
\address{\sloppy
\parbox{0.95\textwidth}{
   Algebraic Geometry Section, Steklov Mathematical Institute of Russian Academy of Sciences,
8 Gubkin str., Moscow 119991, Russia
\hfill
}\bigskip}
\email{avfonarev@mi-ras.ru}
\date{}
\thanks{This work was performed at the Steklov International Mathematical
Center and supported by the Ministry of Science and Higher Education
of the Russian Federation (agreement no. 075-15-2022-265).
The work was supported by the Theoretical Physics and Mathematics Advancement Foundation «BASIS»}

\begin{document}

\begin{abstract}
   We discuss what is known about the structure of the bounded
   derived categories of coherent sheaves
   on Grassmannians of simple algebraic groups.
\end{abstract}

\maketitle

\section{Derived Categories and Semiorthogonal Decompositions}
\subsection{Derived categories}
In the middle of the 20th century, Alexander Grothendieck realized that
it is not sufficient to work with classical derived functors if one wants
to formulate what is now called \emph{coherent} or \emph{Serre--Grothendieck--Verdier}
duality, a relative version of Serre duality. Together with his student
Jean-Louis Verdier, he developed the notion of a derived category.
As it often happens with Grothendieck's ideas, the underlying insight
is beautiful in its simplicity. By that time, it had been known for a while
that homological constructions often operate with various kinds of resolutions
constructed for modules, sheaves of modules, coherent sheaves, etc.
In modern terms, one could say that an abelian category $\CA$
with enough projectives (or injectives) is equivalent to the homotopy category
of projective (injective) resolutions. A resolution is, by definition,
a complex with reasonable terms concentrated only in negative (or positive)
degrees which has only one possibly nontrivial cohomology group in degree 0.
What happens if we remove this restriction on cohomology?
Meanwhile, when we deal with projective and/or injective resolutions,
we treat them up to homotopy equivalence. We do the latter in order
to identify resolutions of the same objects, i.e. resolutions with isomorphic
cohomology. Which equivalence do we put on arbitrary complexes?

The answer that Grothendieck and Verdier give is simple and beautiful. Begin
with an abelian category $\CA$. Consider the category of
complexes $C(\CA)$ (one can consider various versions of this category:
bounded complexes, unbounded complexes, complexes bounded from above or below,
complexes with bounded cohomology, etc.). Now, formally invert all quasi-isomorphisms,
i.e. morphisms of complexes that induce isomorphisms on cohomology. That is
precisely the derived category $D(\CA)$ of $\CA$. While the idea is beautiful
and simple, there is plenty of technical details to fill in. First,
localization (the process of inverting a class of morphisms) is a rather
delicate procedure. Fortunately, it can be done rather painlessly by first
passing to the homotopy category $H(\CA)$: the category of complexes
where morphisms are considered up to homotopy, just like we did for resolutions.
Since homotopic morphisms induce the same morphisms on homology, this first
step does no harm. Now, the second step is to invert quasi-isomorphisms
in the homotopy category. It turns out that the latter step can be done
rather neatly
since in the homotopy category quasi-isomorphisms satisfy the so-called Ore
conditions that one can find in non-commutative algebra.

Once the derived category is constructed, one should naturally ask questions
about its structure. Say, we started with a $\kk$-linear abelian category
over a field $\kk$. While the derived category is naturally $\kk$-linear
and additive, it is no longer abelian. Grothendieck and Verdier came up with
a beautiful notion of a triangulated category, which allows to do much
of the homological algebra just by considering exact triangles
instead of short exact sequences.

Given a smooth projective variety $X$ over a field $\kk$, one defines
its bounded derived category $D(X)$ as the derived category $D(\Coh(X))$
of the category of coherent sheaves $\Coh(X)$ on $X$.
While its construction is simple, for many years derived categories remained
mysterious black boxes with unfathomable internals. It all changed with
the appearance of the paper~\cite{Beilinson1979} by Alexander Beilinson, who in 1978 gave a very
explicit description of the bounded derived category of a projective space.
That work was followed by papers by Mikhail Kapranov, who gave
a similar description of the derived categories of classical Grassmannians
and quadrics.
In both cases it was shown that the bounded derived categories of coherent
sheaves on the corresponding varieties admit full exceptional collections.
Since then, it has been conjectured that the same holds for the derived
categories of all rational homogeneous varieties. In the present survey
we discuss what we know about this conjecture.

\subsection{Semiorthogonal decompositions and exceptional collections}
Let $\CT$ be a triangulated category.
We want to somehow split $\CT$ into smaller pieces. One way of doing it
is using the notion of a semiorthogonal decomposition.

\begin{definition}
   A \emph{semiorthogonal decomposition} of a triangulated category $\CT$
   is a collection of full triangulated subcategories $\CA_1,\CA_2,\ldots,\CA_n\subset \CT$
   such that
   \begin{enumerate}
      \item for all $1\leq i < j \leq n$ and all $X_i\in\CA_i$, $X_j\in \CA_j$
      one has $\Hom_{\CT}(X_j, X_i)=0$,
      \item $\CT$ is the smallest strictly full triangulated subcategory of $\CT$
      containing $\CA_1,\CA_2,\ldots,\CA_n$.
   \end{enumerate}
   A semiorthogonal decomposition is denoted by $\CT=\langle \CA_1,\CA_2,\ldots,\CA_n\rangle$.
\end{definition}

A natural question is whether one can construct a semiorthogonal decomposition
starting with a full triangulated subcategory $\CA\subset \CT$.

\begin{definition}
   A full triangulated subcategory $\CA\subset \CT$ is called \emph{admissible}
   if the inclusion functor $\iota:\CA\to \CT$ admits both
   right and left adjoint functors $\iota^*,\iota^!:\CT\to \CA$.
\end{definition}

If $\CA\subset \CT$ is an admissible full triangulated subcategory,
one can immediately construct
two semiorthogonal decompositions of $\CT$. First, define
the right and left orthogonals to $\CA$ as
\begin{equation*}
   \begin{split}
   \CA^\perp &= \langle X \in \CT \mid \Hom(Y, X)=0\text{ for all }Y\in \CA \rangle,\\
   \lperp{\CA} &= \langle X \in \CT \mid \Hom(X, Y)=0\text{ for all }Y\in \CA \rangle.
   \end{split}
\end{equation*}

\begin{lemma}[\cite{Bondal1990a}]
   If $\CA\subset \CT$ is an admissible full triangulated subcategory,
   then one has semiorthogonal decompositions
   \begin{equation*}
      \CT = \langle \CA^\perp, \CA \rangle
      \quad\text{and}\quad
      \CT = \langle \CA, \lperp{\CA} \rangle.
   \end{equation*}
\end{lemma}

Due to fundamental results of A.~Bondal and M.~Kapranov, see~\cite{Bondal1990a},
all the subcategories that we consider in the present survey
are admissible.

From now on we assume that all our categories are $\kk$-linear,
where $\kk$ is a field. The simplest example of an admissible subcategory
is one generated by an exceptional object.

\begin{definition}
   An object $E\in\CT$ is called \emph{exceptional} if $\Hom(E, E) = \kk$
   and $\Hom(E, E[t]) = 0$ for all $t\neq 0$.
\end{definition}

If $E\in \CT$ is an exceptional object, the smallest
full triangulated subcategory $\langle E \rangle$ generated by $E$
is simply equivalent to the derived category of finite-dimensional vector
spaces over $\kk$, and the latter is equivalent to the category of finite-dimensional
$\ZZ$-graded vector spaces. Alternatively, $\langle E \rangle$ is equivalent
to the bounded derived category $D(\Spec \kk)$ of a point. If a triangulated
category admits a semiorthogonal decomposition such that every component is generated
by an exceptional object, one says that the category admits a full exceptional
collection. Here is a more common definition, where from now on for arbitrary
objects in a triangulated category $\CT$ by $\Ext^i(X,Y)$
we mean $\Hom_{\CT}(X,Y[i])$.

\begin{definition}
   An \emph{exceptional collection} in a triangulated category $\CT$ is a sequence
   of exceptional objects $E_1,E_2,\ldots,E_n$ such that for all $1\leq i < j\leq n$
   one has $\Ext^\bullet(E_j, E_i)=0$.
   An exceptional collection is \emph{full} if there is a semiorthogonal
   decomposition $\CT=\left\langle \langle E_1\rangle, \langle E_2\rangle,\ldots,\langle E_n\rangle\right\rangle$.
\end{definition}

One usually drops the angle brackets for the subcategory generated by an exceptional
object. For instance, one usually denotes by $\langle E_1, E_2,\ldots, E_n\rangle$
the strictly full triangulated subcategory generated by an exceptional collection
$E_1, E_2,\ldots, E_n$.

We are ready to give the first formulation of the pioneering result which
launched a vast area of research in algebraic geometry. We warn the reader
that the statement in the note might seem different, but this is only due
to the fact that the terminology had not been developed at the time of publication.

\begin{theorem}[\cite{Beilinson1979}]\label{thm:bei-1}
   The bounded derived category $D(\mathbb{P}^n)$ of the projective space
   $\mathbb{P}^n$ over a field $\kk$ admits a full exceptional collection
   \begin{equation}\label{eq:pn-ec}
      D(\mathbb{P}^n) = \langle \CO, \CO(1), \ldots, \CO(n) \rangle.
   \end{equation}
\end{theorem}

A similar result was soon proved by Kapranov for classical Grassmannians and quadrics.
Before we present explicit formulations of these results, we can state the main
folklore conjecture that came out of the work of Beilinson and Kapranov.

\begin{conjecture}\label{conj:main}
   Let $\bfG$ be a semisimple algebraic group over an algebraically closed
   field $\kk$ of characteristic 0, and let $\bfP\subset \bfG$ be a parabolic
   subgroup. There is a full exceptional collection in the bounded
   derived category of coherent sheaves $D(\bfG/\bfP)$.
\end{conjecture}

Let us make a few remarks on Conjecture~\ref{conj:main}. First, from the structure
theory of semisimple algebraic groups, parabolic reduction and some generalities
on derived categories, it is rather easy to reduce the statement to the case when
$\bfG$ is simple and $\bfP$ is maximal parabolic. Varieties of the form $\bfG/\bfP$
are often called \emph{generalized Grassmannians}, and we restrict our
attention to them. Next, one can impose
additional conditions on the exceptional objects. Since $\bfG/\bfP$ comes
with a natural action of $\bfG$, it is natural to ask for the collection
to consist of $\bfG$-equivariant objects. However, that is a harmless request
since it was shown
by A.~Polishchuk in~\cite[Lemma~2.2]{Polishchuk2011} that any exceptional object
in $D(\bfG/\bfP)$ admits a~$\bfG$-equivariant structure.
Another restriction one can impose is to ask that all the objects are pure
sheaves and not complexes (thus equivariant vector bundles). In the strongest
form one could ask for extra relations between the exceptional objects.

\begin{definition}
   An exceptional collection $E_1,E_2,\ldots,E_n$ is called \emph{strong}
   if for all $1\leq i < j\leq n$ one has $\Ext^t(E_i, E_j)=0$ for all
   $t\neq 0$.
\end{definition}

The strongest version of Conjecture~\ref{conj:main} for generalized Grassmannians
could be the following.

\begin{conjecture}\label{conj:gr}
   Let $\bfG$ be a simple algebraic group over an algebraically closed
   field $\kk$ of characteristic 0, and let $\bfP\subset \bfG$ be a maximal
   parabolic subgroup. There is a full strong exceptional collection in the bounded
   derived category of coherent sheaves $D(\bfG/\bfP)$ consisting of vector bundles.
\end{conjecture}

In the rest of the survey we discuss what is known about Conjecture~\ref{conj:gr}.
Before we begin, we need to discuss some other general notions from the theory
of derived categories.

\subsection{(Graded) dual exceptional collections}
The notion of a dual exceptional collection could have appeared in the very
same paper by Beilinson~\cite{Beilinson1979} and is analogous to the notion
of a dual basis. Formally, it was studied by A.~Bondal in~\cite{Bondal1990}.
However, since then several conventions related to shifts of objects have been
used in literature. We will take a slightly alternative route and work
with the definitions introduced in~\cite{Fonarev2023}.

First, it will be convenient to work with exceptional collections indexed
by a partially ordered set (poset). Many of our collections will be naturally
indexed by some posets of Young diagrams.
\begin{definition}
   An exceptional collection indexed by a poset $(\CP, \preceq)$ is
   a collection of exceptional
   objects $\{ E_x \}_{x\in\CP}$ such that
   $\Ext^\bullet(E_x, E_y)=0$ unless $x\preceq y$.
\end{definition}

If $\CP=\{1,2,\ldots,n\}$ with its total ordering, then we get the usual notion
of an exceptional collection. However, if $\CP$ has no comparable elements, then
an exceptional collection indexed by $\CP$ consists of pairwise orthogonal objects.

Let us now assume that $\CP$ is finite and \emph{graded}: there exists
a function $|-|:\CP\to \ZZ_{\geq 0}$ such that on each connected component
of $\CP$ it attains value 0, and for any pair $y,x\in\CP$ such that
$y$ covers $x$ one has $|y| = |x|+1$.
\begin{lemma}{{\cite[Lemma~2.5]{Fonarev2023}}}
   Let $\langle E_x\mid x\in\CP \rangle$ be an exceptional collection
    indexed by a finite graded poset $\CP$. For any $y\in \CP$ there exists
    a unique (up to isomorphism) object $E^\circ_y\in \langle E_x\mid x\in\CP \rangle$ such that
    \begin{enumerate}
        \item $\Ext^\bullet(E_x, E^\circ_y)=0$ for all $x\neq y$,
        \item $\Ext^\bullet(E_y, E^\circ_y)=\kk[-|y|]$.
    \end{enumerate}
    The objects $E^\circ_y$ form an exceptional collection
    with respect to the opposite poset $\CP^\circ$.
    This collection is~called the \emph{graded left dual,} and
    $\langle E^\circ_y\mid y\in\CP^\circ \rangle=\langle E_x\mid x\in\CP \rangle$. 
\end{lemma}

When $\CP$ is linearly ordered, the definitions of the graded left dual exceptional
collection and the left dual exceptional collection from~\cite{Bondal1990} agree.
\begin{remark}
   The usual definition of a graded poset simply requires existence of a grading
   function. That is, a function $\nu:\CP\to\ZZ$ such that $\nu(y)=\nu(x)$
   whenever $y$ covers $x$. For instance, the opposite poset $\CP^\circ$
   is graded if and only if $\CP^\circ$ is. If one wants to define the
   \emph{graded right dual} to a graded exceptional collection, then
   one has to take a grading function different from $|-|$.
   Since graded right duals do not appear in the present text, we leave
   the details to the reader.
\end{remark}

\subsection{Lefschetz decompositions}
In many cases the collections that have been constructed in the derived
categories of generalized Grassmannians are Lefschetz.
We begin with the notion of a Lefschetz semiorthogonal decomposition,
originally introduced by A.~Kuznetsov in the context of Homological Projective
Duality. For simplicity, fix a smooth projective variety $X$ with a very
ample line bundle $\CO(1)$.

\begin{definition}[\cite{Kuznetsov2007}]
   A \emph{Lefschetz decomposition} of $D(X)$ is a collection of full
   triangulated subcategories $D(X)\supset \CA_0\supset \CA_1\supset\cdots\supset\CA_{m-1}$
   such that there is a semiorthogonal decomposition
   \begin{equation*}
      D(X) = \langle \CA_0, \CA_1(1), \ldots, \CA_{m-1}(m-1) \rangle,
   \end{equation*}
   where $\CB(i)$ denotes the image of a subcategory $\CB\subset D(X)$
   under the autoequivalence given by the tensor product $-\otimes \CO(i)$.
   If, in addition, $\CA_0=\CA_1=\cdots=\CA_{m-1}$, then the decomposition
   is called \emph{rectangular}.
\end{definition}

Lefschetz decompositions are particularly nice. Imagine, one has found
a rectangular Lefschetz decomposition whose initial block $\CA_0$
is generated by a full exceptional collection. Then it extends to
a~full exceptional collection in the whole derived category:
one should simply add the appropriate twists of the~exceptional objects.
There are particularly interesting cases when one has a natural Lefschetz
decomposition that is not rectangular, thus the following definition.

\begin{definition}[\cite{Fonarev2013}]
   A \emph{Lefschetz basis} is an exceptional collection $(E_1,E_2,\ldots, E_n)$
   together with a~function $o:\{1,2,\ldots,n\}\to \ZZ_{> 0}$ called
   the \emph{support function} such that
   the collection of subcategories $\CA_i = \langle E_j \mid i < o(j) \rangle$
   forms a Lefschetz decomposition of $D(X)$.
   A \emph{Lefschetz exceptional collection} is a~Lefschetz basis
   whose support function is weakly decreasing.
\end{definition}

Remark that the definition of a Lefschetz basis immediately extends to
exceptional collections graded by posets.

Assume further that the variety $X$
is such that its canonical bundle $\omega_X\simeq \CO(-m)$ for some positive $m>0$.
Then it immediately follows from Serre duality that $m$ is the maximal
number of blocks in any Lefschetz decomposition of $D(X)$.

\begin{definition}[\cite{KuznetsovSmirnovGr2020}]
   Assume that $\omega_X\simeq \CO(-m)$. The \emph{residual category}
   of a Lefschetz decomposition $D(X)=\langle \CA_0, \CA_1(1), \ldots, \CA_{m-1}(m-1)\rangle$
   is defined as $\langle \CA_{m-1}, \CA_{m-1}(1), \ldots, \CA_{m-1}(m-1)\rangle^\perp$,
   which is the left orthogonal to its \emph{rectangular part} $\langle \CA_{m-1}, \CA_{m-1}(1), \ldots, \CA_{m-1}(m-1)\rangle$.
\end{definition}

\section{Projective Spaces and Classical Grassmannians}
\subsection{Projective spaces}
The problem of constructing a full exceptional collection consists essentially
of two parts. First, one should find an exceptional collection of appropriate
length. Indeed, if a triangulated category admits a full exceptional collection,
then its Grothendieck group $K_0$ is finitely generated and free. Since
the rank of $K_0$ can often be computed via geometric means, we know what
the length of a~full exceptional collection should be if there is one at all.
Once the objects are found (which seems to be an~art on its own), one needs
to check that they are exceptional and verify semiorthogonality. The latter
two part are essentially cohomological computations, and there are plenty
of tools developed specifically for these purposes.
Next, one needs to somehow show that the collection is full, and in each case
this seems to be a much harder problem.

Let us see how Beilinson dealt with the two tasks in~\cite{Beilinson1979}.
To align our notation with the more general case of classical Grassmannians,
let us fix a vector space $V$ over $\kk$ of dimension $n$. We follow
the anti-Grothendieck convention under which $\PP(V)$ parametrizes
one-dimensional subspaces in $V$. Then $\PP(V)$ is of dimension $(n-1)$,
and Theorem~\ref{thm:bei-1} states that the derived category $D(\PP(V))$
admits a full exceptional collection
\begin{equation}\label{eq:bei-ec}
\langle \CO, \CO(1), \ldots, \CO(n-1)\rangle.
\end{equation}
It is easy to check that the collection~\eqref{eq:bei-ec} is exceptional.
Indeed, a classical computation of cohomology of line bundles on
a projective space implies that
\begin{equation}
   \Ext^\bullet(\CO(j), \CO(i)) = H^\bullet(\PP(V), \CO(i-j)) =
   \begin{cases}
      S^{i-j}V^* & \text{if } 1\leq j\leq i\leq n-1, \\
      0 & \text{if } 1\leq i < j\leq n-1.
   \end{cases}
\end{equation}
We thus simultaneously check both exceptionality of the objects from~\eqref{eq:bei-ec}
and semiorthogonality.

Fullness was checked by Beilinson via a trick that is now called the \emph{resolution
of diagonal} argument.
For an~arbitrary smooth projective variety one may consider the diagram
\begin{equation}
\begin{tikzcd}
   & X \arrow[d, "\Delta"] \arrow[ddr, bend left, "id"] \arrow[ddl, bend right, "id"'] & \\
   & X\times X \arrow[dr, "q"] \arrow[dl, "p"'] & \\
   X & & X,
\end{tikzcd}
\end{equation}
where $p$ and $q$ are the projections on the first and second multiple,
respectively, and $\Delta$ denotes the diagonal embedding.
With the use of the projection formula, the identity functor from $D(X)$ to itself may be cleverly rewritten
as the composition
\begin{equation}\label{eq:pn-id}
   Id_{D(X)}\simeq q_*\!\left(p^*(-)\otimes \Delta_*\CO_X\right),
\end{equation}
where, unless mentioned otherwise, all the functors are derived.
Now, if one finds a good enough resolution for $\Delta_*\CO_X$ in
$D(X\times X)$, then one can deduce statements about generation in $D(X)$.

Let us return to the case $X=\PP(V)$. The structure sheaf of the diagonal
$\Delta_*\CO$ has a Koszul resolution of the form
\begin{equation}\label{eq:bei-koszul}
   0\to p^*\!\left(\Omega^{n-1}(n-1)\right)\otimes q^*\!\left(\CO(-n+1)\right)
   \to\cdots \to p^*\!\left(\Omega^{1}(1)\right)\otimes q^*\!\left(\CO(-1)\right)
   \to \CO\to \Delta_*\CO \to 0,
\end{equation}
where $\Omega^i = \Lambda^i\Omega^1$ is the locally free sheaf of differential
$i$-forms. From the projection formula and~\eqref{eq:pn-id}  we conclude that
every object $F\in D(\PP(V))$ belongs to the subcategory generated by
\begin{equation*}
   \begin{split}
   q_*\!\left(p^*(F)\otimes p^*\!\left(\Omega^{i}(i)\right)\otimes q^*\!\left(\CO(-i)\right)\right) & \simeq
   q_*\!\left(p^*\!\left(F\otimes \Omega^{i}(i)\right)\otimes q^*\!\left(\CO(-i)\right)\right) \\
   & \simeq q_*p^*\!\left(F\otimes \Omega^{i}(i)\right)\otimes \CO(-i) \\
   & \simeq R\Gamma(\PP(V), F\otimes \Omega^{i}(i))\otimes_{\kk}\CO(-i).
   \end{split}
\end{equation*}
We thus see that the objects $(\CO(-n+1), \ldots, \CO(-1), \CO)$ generate
$D(\PP(V))$. Using the (derived) duality anti-autoequivalence, we deduce that
the objects from~\eqref{eq:bei-ec} generate $D(\PP(V))$ as well.

Let us make a few remarks. First, on the projective space there is a very
special Euler short exact sequence
\begin{equation}\label{eq:pn-euler}
   0\to \Omega^1(1) \to V^* \to \CO(1)\to 0,
\end{equation}
where by $V^*$ we denote the trivial bundle with fibre $V^*$ (similarly, we will denote by $V$
the trivial bundle with fibre $V$). If we denote
by $\CL=\CO(-1)$ the tautological line bundle on $\PP(V)$, then we can
rewrite~\eqref{eq:pn-euler}~as
\begin{equation}\label{eq:pn-upeprp}
   0\to \CL^\perp \to V^* \to \CL^*\to 0,
\end{equation}
where $\CL^\perp \simeq (V/\CL)^*$.
The latter presentation will be of great importance for Grassmannians.
Second, the objects
\begin{equation}\label{eq:bei-dual}
   \left\langle\Omega^{n-1}(n-1), \ldots, \Omega^{2}(2), \Omega^{1}(1), \CO\right\rangle
   =
   \langle\Lambda^{n-1}\CL^\perp, \ldots, \Lambda^2\CL^\perp, \CL^\perp, \CO\rangle
\end{equation}
also form a full exceptional collection in $D(\PP(V))$ which
turns out to be left dual to the collection~\eqref{eq:bei-ec}
(we will see this later in a more general setting of classical Grassmannians).
It is not a coincidence that the terms of the resolution~\eqref{eq:bei-koszul}
are of the form $E^\circ_x\boxtimes (E_x)^*$, where $(-)^*$ denotes
the usual (derived) dual.

Unfortunately, the resolution of diagonal argument is rather hard to apply
to cases other than projective spaces and classical Grassmannians
(we will see how it works for the latter in a moment).
Another way of showing fullness of Beilinson's collection relies on a simple
lemma first proved by Dmitri~Orlov. Recall that a \emph{classical generator}
of a triangulated category $\CT$ is an object $E\in\CT$ such that
$\CT$ coincides with the smallest strictly full triangulated subcategory
of $\CT$ which contains $E$ and is closed under direct summands.
For instance, an exceptional collection $\langle E_1, E_2, \ldots, E_r\rangle$
is full if and only if $E=\oplus E_i$ is a classical generator of $\CT$.
For convenience, we formulate the lemma for smooth projective varieties.

\begin{theorem}[{\cite[Theorem~4]{Orlov2009}}]\label{lm:orlov}
   Let $X$ be a smooth projective variety of dimension $d$, and let $\CL$
   be a~very ample line bundle on $X$. For any integer $t$ the object
   $\CE=\bigoplus_{i=t}^{t+d}\CL^i$ is a classical generator of $D(X)$.
\end{theorem}

Let us try to use Theorem~\ref{lm:orlov} for Beilinson's exceptional collection~\eqref{eq:bei-ec}.
Putting $t=0$, we see that $D(\PP(V))$ is classically generated by
$\CO\oplus\CO(1)\oplus\cdots\oplus\CO(n)$, which is one twist more than
the sum of the~objects in~\eqref{eq:bei-ec}: $\CO\oplus\CO(1)\oplus\cdots\oplus\CO(n-1)$.
Thus, to show that Beilinson's collection is full it is enough to show that $\CO(n)$ belongs to the subcategory
generated by $\CO, \CO(1), \ldots, \CO(n-1)$. Consider the regular
nowhere vanishing section $s$ of $V(1)$ which corresponds to the identity
map in
\[
\Gamma(\PP(V), V(1))\simeq V\otimes V^* \simeq \Hom(V, V).
\]
Since it is nowhere vanishing, the Koszul resolution of its zero locus
is simply an exact complex of the form
\begin{equation}\label{eq:pn-koszul}   
   0\to \Lambda^{n}V^*(-n)\to \cdots \to \Lambda^2V^*(-2)\to V^*(-1)\to \CO\to 0.
\end{equation}
Twisting the complex~\eqref{eq:pn-koszul} by $\CO(n)$, we get a complex of the form
\begin{equation}\label{eq:pn-koszul-tw}   
   0\to \Lambda^{n}V^*\to \cdots \to \Lambda^2V^*(n-2)\to V^*(n-1)\to \CO(n)\to 0.
\end{equation}
Treating~\eqref{eq:pn-koszul-tw} as a resolution for it rightmost term $\CO(n)$,
we see that $\CO(n)\in \langle \CO, \CO(1), \ldots, \CO(n-1) \rangle$.
Finally, $\CO\oplus\CO(1)\oplus\cdots\oplus\CO(n)\in \langle \CO, \CO(1), \ldots, \CO(n-1) \rangle$,
so by Theorem~\ref{lm:orlov} Beilinson's collection is full.

The argument presented above can be reformulated as the following proposition.

\begin{proposition}\label{prop:full}
   Let $X$ be a smooth projective variety, let $\CL$ be a very (anti-)ample line bundle on $X$,
   and let $\langle E_1, E_2,\ldots, E_r \rangle\subset D(X)$
   be an exceptional collection such that $\CO\in \langle E_1, E_2,\ldots, E_r \rangle$.
   If for all $i=1,\ldots, r$ the object $E_i\otimes \CL$ belongs to
   $\langle E_1, E_2,\ldots, E_r \rangle$, then the collection is full.
\end{proposition}
\begin{proof}
   By induction, the category $\CT=\langle E_1, E_2,\ldots, E_r\rangle$ contains $\CL^i$ for all $i\geq 0$.
   By Theorem~\ref{lm:orlov}, $\CT$ contains the classical generator $\oplus_{i=0}^{\dim X}\CL$.
   Thus, $\CT=D(X)$.
\end{proof}

To conclude our discussion of projective spaces, we turn to the relative case established by
D.~Orlov.

\begin{theorem}[\cite{Orlov1992}]\label{thm:orlov-proj}
   Let $X$ be a smooth projective variety, and let $\CV$ be a vector bundle on $X$ of rank~$n$.
   Consider the projectivization $\pi:\PP_X(\CV)\to X$, and denote by $\CO_{\CV}(1)$ the relative dual tautological line bundle.
   There is a semiorthogonal decomposition
   \[
      D(\PP_X(\CV)) = \langle D(X), D(X)(1), \ldots, D(X)(n-1) \rangle,
   \]
   where $D(X)(i)$ denotes the subcategory $\pi^*D(X)\otimes \CO_{\CV}(i)$. In particular,
   if $D(X)$ admits a full exceptional collection, then so does $D(\PP_X(\CV))$.
\end{theorem}

Since partial flag varieties of the form $\Fl(1, \ldots, i; V)$ and $\Fl(j, \ldots n-1; V)$
are iterated projectivizations of vector bundles, we deduce the following.

\begin{corollary}\label{cor:fl-1}
   Let $V$ be a vector space for dimension $n$. For all $1\leq i,j\leq n-1$ the derived categories
   of the partial flag varieties $\Fl(1, \ldots, i; V)$ and $\Fl(j, \ldots n-1; V)$ admit
   full exceptional collections consisting of line bundles.
\end{corollary}
Since more general flag varieties are iterated Grassmannian bundles over Grassmannians,
we will need a little more to show that the derived categories of all of them admit
full exceptional collections. To prove the latter will need
the following result concerning fibrewise full exceptional collections.

\begin{theorem}[\cite{Samokhin2007}]\label{thm:sam}
   Let $\pi: X\to S$ be a flat proper morphism of smooth schemes $X$ and $S$ over $\kk$.
   Assume that the objects $E_1, E_2, \ldots, E_r\in D(X)$ are such that for any
   closed point $s\in S$ restrictions to the fibre $X_s$ of $E_i$ form a full exceptional
   collection in $D(X_s)$. Then there is a semiorthogonal decomposition
   \[
      D(X) = \langle \pi^*D(S)\otimes E_1, \pi^*D(S) \otimes E_2, \ldots, \pi^*D(S)\otimes E_r \rangle.
   \]
   In particular, if $D(S)$ admits a full exceptional collection, then so does $D(X)$.
\end{theorem}

Theorem~\ref{thm:sam} has some immediate consequences. The most obvious one is that
given smooth projective varieties $X$ and $Y$ whose derived categories
admit full exceptional collections $D(X)=\langle E_1, E_2, \ldots, E_r\rangle$ and
$D(Y) = \langle F_1, F_2, \ldots, F_s\rangle$, the bounded derived category of their
product admits a full exceptional collection
$D(X\times Y) = \langle E_i\boxtimes E_j \mid i=1,\ldots,r;\ j=1,\ldots, s\rangle$.
Remark that the latter collection is most naturally indexed by product of the partially ordered sets
$\{1,\ldots, r\}\times \{1,\ldots, s\}$.

\subsection{Classical Grassmannians} Beilinson's theorem was generalized to classical Grassmannians
by Mikhail~Kapranov in his 1984 paper~\cite{Kapranov1984}. We begin with a general setup.
As before, let $V$ be a vector space of dimension $n$ over an algebraically closed field $\kk$
of characteristic zero
(for results concerning arbitrary characteristic
we refer the reader to the work~\cite{Buchweitz2015, Efimov2017}).
Fix an integer $1\leq k \leq n-1$ and consider the Grassmannian $\Gr(k, V)$ of $k$-dimensional
subspaces in $V$. It comes with a rank $k$ tautological subbundle $\CU$. One has a short exact sequence
of vector bundles
\begin{equation}\label{eq:gr-ses}
   0\to \CU^\perp\! \to V^*\!\to \CU^*\to 0,
\end{equation}
which in the case $k=1$ specializes to~\eqref{eq:pn-upeprp}.

Denote by $\YD_{h, w}$ the set of Young diagrams of height at most $h$ and width
at most $w$, which can be identified with the set of sequences
\[
   \YD_{h, w} = \{ \lambda \in \ZZ^h \mid w \geq \lambda_1\geq \lambda_2\geq \cdots \geq \lambda_h\geq 0 \}.
\]
The set $\YD_{h,w}$ comes with a natural partial order given by containment: for $\lambda,\mu \in \YD_{h,w}$
put
\[
\lambda \succeq \mu \quad \text{if}\quad \lambda_i \geq \mu_i \text{ for all } i=1,\ldots, h.
\]
Simultaneously, there is a complete lexicographical order on $\YD_{h,w}$ induced by
the corresponding order on~$\ZZ^h$, which we denote by $\geq$. For a Young diagram $\lambda$
we denote by $|\lambda|=\sum\lambda_i$ its size, and by $\lambda^T$ its transpose.
The \emph{width} of $\lambda$ is defined as $\ydw(\lambda) = \lambda_1$, and the \emph{height}
of $\lambda$ is equal to $\ydh(\lambda) = \ydw(\lambda^T)$.

Given a vector bundle $\CE$ and a Young diagram $\lambda$, we denote by $\Sigma^\lambda\CE$
the result of the application of the corresponding Schur functor. We follow the convention
under which $\Sigma^{(t)}=S^t$ is the symmetric power functor. We are ready to formulate
Kapranov' result. In the following theorem we identify the posets $\YD_{k,n-k}$ and
$\YD_{n-k,k}$ via transposition of diagrams.

\begin{theorem}[\cite{Kapranov1984}]\label{thm:kapranov}
   The bounded derived category of $\Gr(k, V)$ admits a full exceptional collection
   indexed by the poset $\YD_{k, n-k}$:
   \begin{equation}\label{eq:gr-ec}
      D(\Gr(k, V)) = \left\langle \Sigma^\lambda\CU^* \mid \lambda\in\YD_{k, n-k} \right\rangle.
   \end{equation}
   Its graded left dual is given by
   \begin{equation}\label{eq:gr-ec-dual}
      D(\Gr(k, V)) = \left\langle \Sigma^{\mu}\CU^\perp \mid \mu\in\YD_{n-k, k} \right\rangle.
   \end{equation}
   That is, for $\lambda\in\YD_{k,n-k}$ and $\mu\in\YD_{n-k,k}$ one has
   \begin{equation*}
      \Ext^\bullet(\Sigma^\lambda\CU^*, \Sigma^\mu\CU^\perp) = \begin{cases}
         \kk[-|\lambda|] & \text{if }\lambda = \mu^T, \\
         0 & \text{otherwise}.
      \end{cases}
   \end{equation*}
\end{theorem}

When $k=1$, Theorem~\ref{thm:kapranov} produces Beilinson's exceptional
collection on $\PP(V)=\Gr(1, V)$. Indeed, one has $\CO(1) = \CU^*$, $\Omega^1(1)\simeq \CU^\perp$,
and $\YD_{1,n-1} = \{(i) \mid i=0,\ldots, n-1\}$.
For each $0\leq i\leq n-1$ there are isomorphisms $\Sigma^{(i)}\CO(1) = S^i\CO(1)\simeq \CO(i)$
and $\Sigma^{(i)^T}\CU^\perp = \Lambda^i\left(\Omega^1(1)\right)\simeq \Omega^i(i)$.
Thus, Theorem~\ref{thm:kapranov} is a~generalization of Theorem~\ref{thm:bei-1}.

The original proof of Theorem~\ref{thm:kapranov} is analogous to Beilinson's
proof of Theorem~\ref{thm:bei-1}. First, one needs to check that the objects from~\eqref{eq:gr-ec}
are exceptional and semiorthogonal. Since all of them are irreducible equivariant vector bundles,
one can apply the celebrated Borel--Bott--Weil theorem together with the Littlewood--Richardson rule.
In order to show fullness, Kapranov applies a resolution of diagonal argument. Consider the~diagram
\begin{equation*}
   \begin{tikzcd}
      & \Gr(k, V)\times \Gr(k, V) \arrow[dr, "q"] \arrow[dl, "p"'] & \\
      \Gr(k, V) & & \Gr(k, V).
   \end{tikzcd}
\end{equation*}
The composition of natural morphisms
\[
   p^*\CU^\perp \to V^* \to q^*\CU^*
\]
vanishes precisely along the diagonal, which can be described as the vanishing locus
of the corresponding section $s\in\Gamma(\Gr(k, V)\times \Gr(k, V),\, (V/\CU)\boxtimes \CU^*)$,
where $\CE\boxtimes\CF = p^*\CE\otimes q^*\CF$ and $(V/\CU)\simeq (\CU^\perp)^*$.
The Koszul resolution of $\Delta_*\CO$ has terms of the form
\[
   \Lambda^i\!\left(\CU^\perp\!\boxtimes\CU\right) \simeq \bigoplus_{|\lambda|=i} \Sigma^{\lambda^T}\!\CU^\perp\boxtimes\Sigma^\lambda\CU,
\]
where the isomorphism follows from~\cite[Corollary~2.3.3]{Weyman2003}.
Remark that the product
$\Sigma^{\lambda^T}\!\CU^\perp\boxtimes\Sigma^\lambda\CU$ is nonzero
if and only if $\lambda\in\YD_{k, n-k}$ since $\rk(\CU) = k$ and $\rk(\,\CU^\perp) = n-k$.
The rest of the argument goes just as the one we presented for projective spaces.

The exceptional collection~\eqref{eq:gr-ec} is defined in an obvious way in the relative setting.
Applying Theorem~\ref{thm:sam}, we deduce the following.

\begin{theorem}\label{thm:gr-rel}
   Let $\CV$ be a vector bundle of rank $n$ on a smooth projective variety $X$, and let $1\leq k\leq n-1$
   be an integer. Consider the relative Grassmannian bundle
   $\pi:\Gr_X(k, \CV)\to X$, and let $\CU\subset \pi^*\CV$ denote the relative tautological bundle.
   Then there is a (graded) semiorthogonal decomposition of the form
   \[
      D(\Gr_X(k, \CV)) = \left\langle \pi^* D(X)\otimes \Sigma^\lambda\CU^* \mid \lambda\in\YD_{k, n-k} \right\rangle.
   \]
   In particular, if $D(X)$ has a full exceptional collection (consisting of vector bundles),
   then $D(\Gr_X(k, \CV))$ also has a full exceptional collection (consisting of vector bundles).
\end{theorem}

\begin{corollary}
   Let $X$ be a smooth projective variety, which can be obtained as an iterated Grassmannian
   bundle over a smooth projective variety $S$ such that $D(S)$ admits a full exceptional collection.
   Then so does $D(X)$.
\end{corollary}

Since for all integer sequences $1\leq i_1 < i_2 < \cdots < i_s \leq n-1$ the partial flag variety
$\Fl(i_1,\ldots, i_s; V)$ is an iterated Grassmannian bundle, say, over the Grassmannian $\Gr(i_1, V)$,
we conclude the following.

\begin{corollary}\label{cor:type-a}
   All (partial) flag varieties over an algebraically closed field of characteristic zero
   admit full exceptional collections consisting of equivariant vector bundles.
\end{corollary}

Corollary~\ref{cor:type-a} completely establishes Conjecture~\ref{conj:main} in type A.

\subsection{Staircase complexes}\label{ssec:stair} A class of long exact sequences of equivariant
vector bundles was introduced in~\cite{Fonarev2013} in order to study Lefschetz
exceptional collections (bases) in $D(\Gr(k, V))$. In the present section we describe these
sequences, which are called \emph{staircase complexes} and give a construction which is
independent on the existence of Kapranov's exceptional collection (this construction
appeared, unpublished, in th e author's PhD thesis).

For any positive integers $h,w>0$ such that $h+w=n$ the set $\YD_{h,w}$ carries
a cyclic group action: fix a generator $g$ of the cyclic group $\ZZ/n\ZZ$ and put
\begin{equation}\label{eq:cyclic}
   \lambda \mapsto g\cdot\lambda = \begin{cases}
      (\lambda_1+1, \lambda_2+1, \ldots, \lambda_h+1) & \text{if } \ydw(\lambda) < w, \\
      (\lambda_2, \ldots, \lambda_h, 0) & \text{if } \ydw(\lambda) = w.
   \end{cases}
\end{equation}
For brevity, we denote $\lambda'=g\cdot \lambda$. The easiest way to see that~\eqref{eq:cyclic}
indeed defines the cyclic action is the following. Recall that there is a bijection
between $\YD_{h,w}$ and the set of binary sequences $a_1a_2\ldots a_n\in \{0,1\}$ such that $\sum a_i=h$:
each element of $\YD_{h,w}$ represents an integer path going from the lower left corner
of an $h$ by $w$ rectangle to the top right one, moving one step right or up at a time.
For each step to the right we write $0$, and for each step going up we write $1$.
Under this bijection the cyclic action on $\YD_{h,w}$ translates to the usual cyclic action on binary
sequences, $g\cdot (a_1a_2\ldots a_n) = a_na_1\ldots a_{n-1}$, which obviously preserves the sum of the elements.
For any sequence $\lambda\in\ZZ^h$ and any integer $t$ put
\begin{equation*}
   \lambda(t) = (\lambda_1 + t, \lambda_2 + t, \ldots, \lambda_h + t).
\end{equation*}
If $\lambda$ is weakly decreasing, then it corresponds to a dominant weight
of $\GL_h$, and for any rank $h$ bundle $\CE$ one has
$\Sigma^{\lambda(t)}\CE\simeq \Sigma^{\lambda}\CE\otimes (\det\CE)^{\otimes t}$.
Put
\begin{equation*}
   \bar\lambda = (\lambda_2, \ldots, \lambda_h, 0).
\end{equation*}
Then we can reformulate~\eqref{eq:cyclic} as $\lambda'=\lambda(1)$ as long as $\lambda(1)\in\YD_{h,w}$
and $\lambda'=\bar\lambda$ if $\lambda(1)\notin\YD_{h,w}$.

Let $\lambda\in\YD_{h,w}$ be a diagram of maximal width, i.e. $\ydw(\lambda)=w$, and let $\mu=\lambda^T$.
Remark that the condition on $\ydw(\lambda)$ implies that $\mu_w > 0$.
For $i=1,\ldots,w$, define $\lambda^{(i)}$ according to the following rule:
\begin{equation}\label{eq:circ}
   \begin{split}
      \lambda^{(1)} & = (\mu_1, \mu_2, \ldots, \mu_{w-1}, 0)^T, \\
      \lambda^{(i)} & = (\mu_1, \ldots, \mu_{w-i},\, \mu_{w-i+2}-1, \ldots, \mu_w-1, 0)^T \quad \text{for } 1 < i < w,\\
      \lambda^{(w)} & = (\mu_2-1, \ldots, \mu_w-1, 0)^T.
   \end{split}
\end{equation}
The formulas~\eqref{eq:circ} may look rather mysterious; however, they have a rather simple
combinatorial explanation. Informally speaking, in order to find $\lambda^{(i)}$,
one should take the inner band of unit width along
the border of $\lambda$ (this band is nothing but the skew Young diagram $\lambda/\lambda^{(w)}$)
and cut $i$ rightmost columns of the band out of $\lambda$.
In terms of binary sequences, if $a_1a_2\ldots a_n$ is the sequence corresponding to $\lambda$,
let $n>z_1>z_2>\cdots > z_w\geq 1$ be the indices for which $a_{z_i}=0$. Then $\lambda^{(i)}$
is the diagram corresponding to the binary sequence
\[
   a_1a_2\ldots a_{z_i-1}1a_{z_i+1}\ldots a_{n-1}0
\]
(recall that $\ydw(\lambda)=w$, so $a_n=1$).
We thus get a sequence of subdiagrams
\begin{equation}\label{eq:circ-filtr}
   \lambda \supset \lambda^{(1)}\supset \cdots \supset \lambda^{(w)}.
\end{equation}

\begin{example}
   Consider $\lambda = (4,4,2)\in \YD_{3, 4}$. Then the filtration~\eqref{eq:circ-filtr} becomes
   \[
      \ytableausetup{boxsize=1em}
      \ytableausetup{centertableaux}
      \lambda=
      \ydiagram{4, 4, 2} \supset
      \ydiagram[*(lightgray)]{3+1,3+1}*{4, 4, 2} \supset
      \ydiagram[*(lightgray)]{3+1,2+2}*{4, 4, 2} \supset
      \ydiagram[*(lightgray)]{3+1,1+3,1+1}*{4, 4, 2} \supset
      \ydiagram[*(lightgray)]{3+1,1+3, 2}*{4, 4, 2}\ ,
   \]
   where $\lambda^{(i)}$ is the white part of the corresponding diagram.
\end{example}

\begin{proposition}[{\cite[Proposition~5.3]{Fonarev2013}}]\label{prop:starcase}
   Let $\lambda\in\YD_{k, n-k}$ be a diagram such that $\ydw(\lambda)=n-k$.
   There is a long exact sequence of vector bundles on $\Gr(k, V)$ of the form
   \begin{equation}\label{eq:staricase}
      0\to \Sigma^{\bar\lambda}\CU^*(-1) \to \Lambda^{b_\lambda^{(w)}}V^*\otimes \Sigma^{\lambda^{(w)}}\CU^* \to
      \cdots
      \to \Lambda^{b_\lambda^{(1)}}V^*\otimes \Sigma^{\lambda^{(1)}}\CU^* \to
      \Sigma^\lambda\CU^* \to 0,
   \end{equation}
   where $b_\lambda^{(i)} = |\lambda/\lambda^{(i)}| = |\lambda|-|\lambda^{(i)}|$.
\end{proposition}

\begin{remark}
   The proof of Proposition~\ref{prop:starcase} given in~\cite{Fonarev2013}
   relies on fullness of Kapranov's exceptional collection. Indeed,
   the long exact sequence~\eqref{eq:staricase} explicitly shows how
   the object $\Sigma^{\bar\lambda}\CU^*$ can be generated by the objects
   from Kapranov's collection. As promised in the beginning of the present
   section, below we give a fullness-independent construction of~\eqref{eq:staricase}.
\end{remark}

\begin{proof}[Proof of Proposition~\ref{prop:starcase}]
   Consider the diagram
   \begin{equation*}
      \begin{tikzcd}
         & \Fl(1, k; V) \arrow[dr, "q"] \arrow[dl, "p"'] & \\
         \Gr(k, V) & & \PP(V).
      \end{tikzcd}
   \end{equation*}
   Denote by $\CF^\bullet$ the dual of the exact complex~\eqref{eq:pn-koszul-tw} twisted by $\CO(-n+1)$, which
   is an exact complex on $\PP(V)$ of the form
   \begin{equation}\label{eq:stair-pn}
      0\to \CO(-n+1) \to V(-n+2)\to \Lambda^2V(-n+3)\to \cdots \to \Lambda^{n-1}V\to \CO(1)\to 0,
   \end{equation}
   where $\CF^0 = \CO(1)$.

   Denote by $\CL\subset \CU$ the universal flag on $\Fl(1, k; V)$. Consider the object
   \[
      X^\bullet = p_*\left( q^*\CF^\bullet \otimes \Sigma^{\bar\lambda}(\CU/\CL)\otimes (\det (\CU/\CL))^{-1} \right),
   \]
   where $\Sigma^{\bar\lambda}(\CU/\CL)$ is non-zero since $\ydh(\bar{\lambda}) \leq h-1 = \rk(\CU/\CL)$.
   Since $\CF^\bullet$ is exact, $X^\bullet$ is zero in $D(\Gr(k, V))$. Meanwhile, one has the
   standard spectral sequence
   \begin{equation*}
      E^{ab}_1 = R^bp_*\left( q^*\CF^a \otimes \Sigma^{\bar\lambda}(\CU/\CL)\otimes (\det (\CU/\CL))^{-1} \right)
      \Rightarrow H^{a+b}(X^\bullet).
   \end{equation*}
   By definition, for $a=-n, \ldots, 0$ we get
   \begin{equation*}
      \begin{split}
         E^{ab}_1 & = R^bp_*\left( q^*(\Lambda^{n+a}V(a+1)) \otimes \Sigma^{\bar\lambda}(\CU/\CL)\otimes (\det (\CU/\CL))^{-1} \right) \\
                  & \simeq R^bp_*\left( \Lambda^{n+a}V\otimes \CL^{-a-1} \otimes \Sigma^{\bar\lambda}(\CU/\CL)\otimes (\det (\CU/\CL))^{-1} \right) \\
                  & \simeq \Lambda^{n+a}V(1) \otimes R^bp_*\left(\CL^{-a} \otimes \Sigma^{\bar\lambda}(\CU/\CL) \right),
      \end{split}
   \end{equation*}
   where in the last isomorphism we used the identity $\det\CU = \det(\CU/\CL)\otimes \CL$.

   Since $\Fl(1, k; V)$ is nothing but the projectivization $\PP_{\Gr(k, V)}(\CU)$, we can use
   the Borel--Bott--Weil theorem to compute the higher push-forwards. It turns out, that
   the spectral sequence degenerates into an exact complex, which is dual to~\eqref{eq:staricase}.
\end{proof}

\begin{example}
   A well known example of a staircase complex comes from $\lambda = (n-k)$:
   \[
      0\to \CO(-1) \to \Lambda^{n-k}V^*\otimes \CO \to \cdots
      \to \Lambda^{n-k-i}V^*\otimes S^i\CU^* \to \cdots
      \to V^*\otimes S^{n-k-1}\CU^*\to S^{n-k}\CU^*\to 0.
   \]
   When $k=1$, the only staircase complex is a shifted and dualized Koszul complex,
   the complex dual to~\eqref{eq:stair-pn}.
\end{example}

An immediate application of staircase complexes is an alternative proof of fullness
of Kapranov's exceptional collection.
Indeed, let $\lambda\in\YD_{h, w}$. If $\ydh(\lambda) = h$, then $\lambda(-1)\in\YD_{h, w}$.
Since 
\[
   \Sigma^{\lambda(-1)}\CU^*\simeq\Sigma^{\lambda}\CU^* \otimes (\det\, \CU^*)^{-1}\simeq \Sigma^\lambda\CU^*(-1),
\]
we conclude that
for all such $\lambda$ the object $\Sigma^\lambda\CU^*(-1)$ trivially belongs to the subcategory
generated by the objects from Kapranov's collection.
If $\ydh(\lambda) < h$, then $\lambda = \bar\mu$, where $\mu = (n-k, \lambda_1, \ldots, \lambda_{h-1})$,
and the staircase complex~\eqref{eq:staricase} written for $\mu$ and treated as a right
resolution for $\Sigma^{\lambda}\CU^*(-1)$ shows that $\lambda=\bar\mu$
also belongs to the subcategory generated by Kapranov's collection. Applying Proposition~\ref{prop:full},
we see that Kapranov's collection is full. While this proof gives nothing new for classical Grassmannians,
this method will be generalized below to the case of Lagrangian Grassmannians, and we expect that
it can be used in even greater generality.

\subsection{Lefschetz decompositions for classical Grassmannians}
For a smooth projective Fano variety it is natural to look for a Lefschetz basis (decomposition)
which is as small as possible. There are various ways to interpret ``minimal'': one can ask for
the number of objects in the basis to be minimal or the rectangular part to be as large as possible.
Recall that $\Gr(k, V)$ is a Fano variety whose Picard group is isomorphic to $\ZZ$ and is generated
by $\CO(1)\simeq \det\CU^*$. Its (Fano) index equals $n$: $\omega_{\Gr(k, V)}\simeq \CO(-n)$.

The first example of a minimal (in any sense) Lefschetz basis in $D(\Gr(k, V))$ was constructed
for $k=2$ by A.~Kuznetsov in~\cite{Kuznetsov2008}.

\begin{theorem}[{\cite{Kuznetsov2008}}]\label{thm:ld-gr2}
   If the dimension $n=2t+1$ of $V$ is odd,
   then there is a rectangular Lefschetz decomposition in $D(\Gr(2, V))$
   with the basis
   \[
      (\CO, \CU^*, S^2\CU^*, \ldots, S^{t-1}\CU^*), \qquad o = (n, n, n,\ldots, n).
   \]
   If the dimension $n=2t$ of $V$ is even,
   then there is a Lefschetz decomposition in $D(\Gr(2, V))$
   with the basis
   \[
      (\CO, \CU^*, \ldots, S^{t-2}\CU^*, S^{t-1}\CU^*), \qquad o = (n, n, \ldots, n, n/2).
   \]
\end{theorem}

Recall that the number of objects in any full exceptional collection equals the rank of the Grothendieck group.
Since $\rk K_0(\Gr(k, V)) = \binom{n}{k}$, the minimal possible number of object one can get
in the first block of a Lefschetz decomposition is $\left\lceil\binom{n}{k}/n\right\rceil$.
(Recall that by Serre duality one can not have more blocks than the Fano index of the variety).
For $k=2$ we get the lower bound $\left\lceil\binom{n}{2}/2\right\rceil = \left\lceil \frac{n-1}{2}\right\rceil$,
which is achieved in Theorem~\ref{thm:ld-gr2}.
We also see that, for divisibility reasons, when $n$ is odd, $D(\Gr(2, V))$ can not have a~minimal rectangular
Lefschetz basis.

\begin{remark}\label{rmk:gr2}
   A more common way to write Lefschetz exceptional collections is in matrix form, where
   the objects are read bottom to top and left to right. For instance, when $n=2t+1$ is odd, one
   has a full exceptional collection in $D(\Gr(2, V))$ of the form
   \begin{equation}\label{eq:gr2odd}
      \begin{pmatrix}
         S^{t-1}\CU^* & S^{t-1}\CU^*(1) & \cdots & S^{t-1}\CU^*(2t) \\
         S^{t-2}\CU^* & S^{t-2}\CU^*(1) & \cdots & S^{t-2}\CU^*(2t) \\
         \vdots & \vdots &  & \vdots \\
         \CU^* & \CU^*(1) & \cdots & \CU^*(2t) \\
         \CO^* & \CO^*(1) & \cdots & \CO^*(2t)
      \end{pmatrix},
   \end{equation}
   and when $n=2t$ is even, one has a full exceptional collection in $D(\Gr(2, V))$ of the form
   \begin{equation*}
      \begin{pmatrix}
         S^{t-1}\CU^* & S^{t-1}\CU^*(1) & \cdots & S^{t-1}(t-1) &  \\
         S^{t-2}\CU^* & S^{t-2}\CU^*(1) & \cdots & S^{t-2}(t-1) & \cdots & S^{t-2}\CU^*(2t-1) \\
         \vdots & \vdots &  & \vdots & & \vdots \\
         \CU^* & \CU^*(1) & \cdots & \CU(t-1) & \cdots & \CU^*(2t-1) \\
         \CO^* & \CO^*(1) & \cdots & \CO(t-1) & \cdots & \CO^*(2t-1)
      \end{pmatrix}.
   \end{equation*}
\end{remark}

The argument used in~\cite{Kuznetsov2008} to show fullness of the two collections relies on the following
rather common technique. First, cover your variety by a family of simpler subvarieties.
In this case, one can consider zero loci of linear forms $\phi\in \Gamma(\Gr(k, V), \CU^*)\simeq V^*$,
which are isomorphic to $\Gr(k, \Ker\phi)$. Observe that if all the restrictions of an object $X\in D(\Gr(k, V))$
to the subvarieties vanish, then the object itself is zero. Finally, check that if an object is orthogonal
to all the objects in the collection, the fullness of which one wants to prove, then
its restriction is zero for all choices of $\phi$. Unfortunately, this method is not always easy to apply.

We are about to give a generalization of Theorem~\ref{thm:ld-gr2} to arbitrary $\Gr(k, V)$, where
fullness will be shown with the use of staircase complexes. Actually, we are going to describe
three Lefschetz bases in $D(\Gr(k, V))$. When $k$ and $n$ are coprime, the first two will coincide
and will give minimal rectangular decompositions for $D(\Gr(k, V))$.
When $k$ and $n$ are not coprime, the first one will be slightly larger. However, this is the one
that we know to be full in all cases (the second one is conjecturally full for $k$ and $n$ not coprime).
Finally, the third one is of the same ``shape'' as the second one. It is minimal in the sense
that its rectangular part is maximal possible, and we know that it is full, though the proof has not
appeared in the literature just yet and will be available in~\cite{FonarevUnpublished}.

Or strategy of constructing a Lefschetz basis will always be the same: pick a subset of 
$\YD_{k, n-k}$ and use the corresponding objects from Kapranov's collection as the first block
of the decomposition by assigning appropriate values of the support function.
\begin{definition}
   A diagram $\lambda\in\YD_{h, w}$ is called \emph{upper-triangular} if for all $1\leq i\leq h$
   one has
   \begin{equation}\label{eq:ydu}
      \lambda_i \leq \frac{w(h-i)}{h}.
   \end{equation}
   The subset of upper-triangular diagrams in $\YD_{h, w}$ is denoted by $\YDu_{h,w}\subset\YD_{h, w}$.
\end{definition}

\begin{example}
   We give some examples of sets of upper-triangular diagrams.
   \begin{enumerate}
      \item For $h=1$ the set $\YDu_{1, w}$ always consists of a single diagram: $\YD_{1, w} = \{(0)\}$.
      \item When $h=2$, one has $\YDu_{2, w} = \{(0, 0), (1, 0), \ldots, (\lfloor n/2\rfloor, 0)\}$.
      \item For $h=3$ and $w=3$ one has
      \[
         \YDu_{3,3} = \left\{\emptyset,\ \ydiagram{1},\ \ydiagram{2},\ \ydiagram{1, 1},\ \ydiagram{2, 1} \right\}.
      \]
   \end{enumerate}
\end{example}

For any upper-triangular diagram $\lambda\in\YDu_{w,h}$ let 
\[
   o^{\rm u}(\lambda)=i+w-\lambda_i,
\]
where $i$ is the smallest positive integer for which inequality~\eqref{eq:ydu} becomes an equality
(it is always the case for $i=h$). In other words, $o^{\rm u}(\lambda)$ is the length of the path going along
$\lambda$ from the top right corner of the $h$ by $w$ rectangle until one meets the diagonal again.
In terms of the cyclic group action, $o^{\rm u}(\lambda)$ is the smallest positive integer $i$ such that
$g^i\cdot\lambda\in \YDu_{h, w}$, where $g$ is the chosen generator of $\ZZ/n\ZZ$.

\begin{example}
   Here are some examples of $o^{\rm u}$ for various diagrams.
   \begin{enumerate}
   \item When $w$ and $h$ are coprime, then for any $\lambda\in\YDu_{w, h}$ one has $o^{\rm u}(\lambda) = n$.
   \item When $h=2$ and $w=2t$ one has $o^{\rm}((t, 0)) = t+1 = n/2$.
   \item For $h=3$ and $w=3$ one has the following values of $o^{\rm u}$:
   \[
      o^{\rm u}: \left(\emptyset,\ \ydiagram{1},\ \ydiagram{2},\ \ydiagram{1, 1},\ \ydiagram{2, 1} \right)
      \mapsto
      (6, 6, 2, 4, 2).
   \]
   \end{enumerate}
\end{example}

\begin{theorem}[{\cite[Theorem~4.1]{Fonarev2013}}]\label{thm:gr-lda}
   There is a Lefschetz decomposition of $D(\Gr(k, V))$ with a Lefschetz
   basis given by
   \begin{equation}\label{eq:gr-lda}
      \left\langle \Sigma^\lambda\CU^* \mid \lambda\in \YDu_{k, n-k}\right\rangle
   \end{equation}
   and the support function $o^{\rm u}$.
\end{theorem}

\begin{example}
   Consider $\Gr(3, 6)$. Theorem~\ref{thm:gr-lda} produces a Lefschetz decomposition
   with the basis
   \[
      \left\langle \CO, \CU^*, \Lambda^2\CU^*, S^2\CU^*, \Sigma^{2, 1}\CU^* \right\rangle.
   \]
   The corresponding (full) Lefschetz exceptional collection in $D(\Gr(3, 6))$ is of the form
   \begin{equation}\label{eq:gr36-lda}
      \begin{pmatrix*}[r]
         \Sigma^{2, 1}\CU^* & \Sigma^{2, 1}\CU^*(1) & & & & \\
         S^2\CU^* & S^2\CU^*(1) & & & & \\
         \Lambda^2\CU^* & \Lambda^2\CU^*(1) & \Lambda^2\CU^*(2) & \Lambda^2\CU^*(3) &  & \\
         \CU^* & \CU^*(1) & \CU^*(2) & \CU^*(3) & \CU^*(4) & \CU^*(5) \\
         \CO & \CO(1) & \CO(2) & \CO(3) & \CO(4) & \CO(5)
      \end{pmatrix*}
   \end{equation}
\end{example}

The proof of Theorem~\ref{thm:gr-lda} relies on some intricate homology
computations using the Borel--Bott--Weil theorem and the Littlewood--Richardson
rule for semiorthogonality and staircase complexes for fullness.
Namely, using staircase complexes and induction, it is shown in~\cite{Fonarev2013} that
the objects of Kapranov's collection are generates by the objects given
by the Lefschetz basis~\eqref{eq:gr-lda}.

When $k=2$, Theorem~\ref{thm:gr-lda} recovers Theorem~\ref{thm:ld-gr2}.
When $k$ and $n$ are coprime, Theorem~\ref{thm:gr-lda} produces
a rectangular Lefschetz decomposition which is optimal due to numerical reasons:
the number of objects in the basis equals $\rk K_0(\Gr(k, V)) / n$.
However, when $k$ and $n$ are not coprime, one can find a smaller Lefschetz
decomposition in $D(\Gr(k, V))$. This decomposition agrees with
the one presented in Theorem~\ref{thm:gr-lda} when $(k, n)=1$ and is expected
to be full for arbitrary $k$ and $n$.

Consider again the cyclic action of $\ZZ/n\ZZ$ on $\YD_{w, h}$.
The following observation is a very pleasant and easy exercise.
\begin{lemma}[{\cite[Lemma~3.2]{Fonarev2013}}]
   Every orbit of the cyclic action on $\YD_{w, h}$ contains
   an upper-triangular diagram.
\end{lemma}

For each $\lambda$ denote by $o(\lambda)$ the length of the cyclic orbit
containing $\lambda$.
Our strategy for constructing a Lefschetz basis will be the following:
we want the elements of the basis to be represented by elements of the orbits
of the action, while the support function will be the length of
the corresponding orbit.
One way to pick an element in each orbit is to go for upper-triangular
diagrams. However, by looking at the orbits for $h=3$ and $w=3$ we see
that such representatives are not unique as soon as $h$ and $w$ are not
coprime and $h, w > 2$. Indeed, the corresponding orbits are
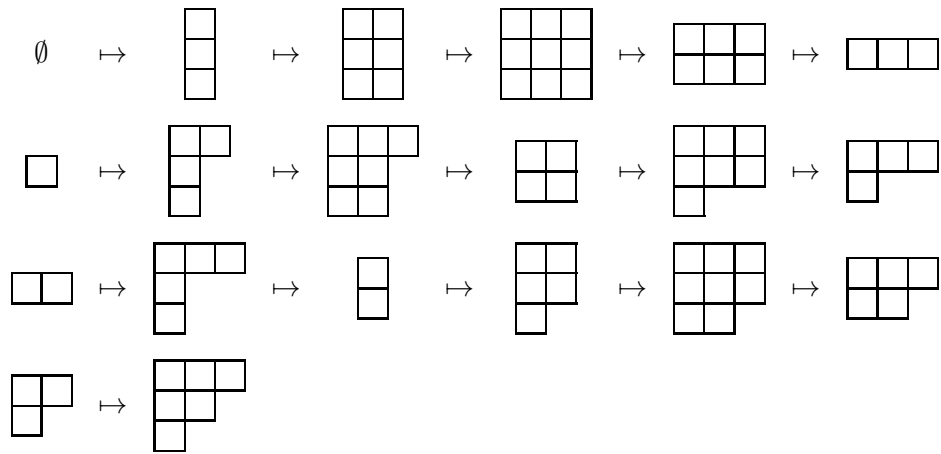
\[
   \begin{array}[b]{ccccccccccc}
      \emptyset & \mapsto & \ydiagram{1,1,1} & \mapsto & \ydiagram{2,2,2} & \mapsto & \ydiagram{3,3,3} & \mapsto & \ydiagram{3, 3} & \mapsto & \ydiagram{3} \\[20pt]
      \ydiagram{1} & \mapsto & \ydiagram{2,1,1} & \mapsto & \ydiagram{3,2,2} & \mapsto & \ydiagram{2,2} & \mapsto & \ydiagram{3, 3, 1} & \mapsto & \ydiagram{3, 1} \\[20pt]
      \ydiagram{2} & \mapsto & \ydiagram{3,1,1} & \mapsto & \ydiagram{1,1} & \mapsto & \ydiagram{2,2,1} & \mapsto & \ydiagram{3, 3, 2} & \mapsto & \ydiagram{3, 2} \\[20pt]
      \ydiagram{2, 1} & \mapsto & \ydiagram{3,2,1} &  &  &  & 
   \end{array}
\]
and we see that both $(2, 0, 0)$ and $(1, 1, 0)$ are upper-triangular elements
lying in the same orbit.
A rather straightforward idea is to pick in each orbit the upper-triangular
element which is lexicographically minimal. Denote by $\YDmu_{h,w}$
the set of such diagrams.

\begin{theorem}[{\cite[Theorem~4.3]{Fonarev2013}}]\label{thm:gr-ldb}
   Let $\CA_i = \langle \Sigma^\lambda \CU^* \mid \lambda \in \YDmu_{k, n-k},\ i < o(\lambda) \rangle$. There is a semiorthogonal decomposition
   \begin{equation}\label{eq:gr-ldb}
      D(\Gr(k, V)) = \left\langle \CR^{\rm mu},\, \CA_0, \CA_1(1), \ldots, \CA_{n-1}(n-1) \right\rangle,
   \end{equation}
   where $\CR^{\rm mu}$ is expected to be zero.
\end{theorem}

\begin{conjecture}[{\cite[Conjecture~4.4]{Fonarev2013}}]\label{conj:gr-ldb}
   The orthogonal $\CR^{\rm mu}$ in Theorem~\ref{thm:gr-ldb} vanishes.
   In other words, $\langle \Sigma^\lambda \CU^* \mid \lambda \in \YDmu_{k, n-k}\rangle$
   is a Lefschetz basis with the support function $o$.
\end{conjecture}

When $(k, n)=1$, one has $\YDmu_{k, n-k} = \YDu_{k, n-k}$, so we deduce from
Theorem~\ref{thm:gr-lda} that Conjecture~\ref{conj:gr-ldb} holds: the two bases
coincide. The same holds true when $k=2$ or $n-k=2$. In the next simplest
case, when $k=3$ and $n=6$, Conjecture~\ref{conj:gr-ldb} was established
in \cite[Proposition~5.7]{Fonarev2013}. In particular, one gets
a~full Lefschetz exceptional collection
\begin{equation}\label{eq:gr35-ldb}
   \begin{pmatrix*}[r]
      \Sigma^{2, 1}\CU^* & \Sigma^{2, 1}\CU^*(1) & & & & \\
      \Lambda^2\CU^* & \Lambda^2\CU^*(1) & \Lambda^2\CU^*(2) & \Lambda^2\CU^*(3) & \Lambda^2\CU^*(4) & \Lambda^2\CU^*(5) \\
      \CU^* & \CU^*(1) & \CU^*(2) & \CU^*(3) & \CU^*(4) & \CU^*(5) \\
      \CO & \CO(1) & \CO(2) & \CO(3) & \CO(4) & \CO(5)
   \end{pmatrix*}
\end{equation}
(compare with~\eqref{eq:gr36-lda}).
In~\cite{KuznetsovSmirnovGr2020} A.~Kuznetsov and Maxim~Smirnov
used decomposition~\eqref{eq:gr-ldb} in order to study a certain generalization
of Dubrovin's conjecture and
established Conjecture~\ref{conj:gr-ldb} for $k=3$, see~\cite[Proposition~A.1]{KuznetsovSmirnovGr2020}. It remains open in the other cases.

Finally, we present our third Lefschetz decomposition. It has the same
shape as the conjecturally full one constructed in Theorem~\ref{thm:gr-ldb},
and one can prove that it is full, although the proof has not appeared in the literature
just yet. Once again, we want to choose a representative in each orbit
of the cyclic action on $\YD_{k, n-k}$ and use the length of the orbit
$o$ as the support function (thus the same shape of the basis).
Let us drop the upper-triangular condition and simply pick the lexicographically
minimal element in each orbit. Denote the set of such elements by $\YDm_{k, n-k}$.

\begin{theorem}[{\cite{FonarevUnpublished}}]\label{thm:gr-ldm}
   There is a Lefschetz decomposition of $D(\Gr(k, V))$ with a Lefschetz
   basis given by
   \begin{equation}\label{eq:gr-ldm}
      \left\langle \Sigma^\lambda\CU^* \mid \lambda\in \YDm_{k, n-k}\right\rangle
   \end{equation}
   and the support function $o$.
\end{theorem}

Remark that the sets $\YDm_{w,h}$ and $\YDmu_{w, h}$ are different even
when $h$ and $w$ are coprime.
Indeed, the diagram $(1, 1, 1, 0)\in \YDm_{4, 3}$ is minimal, but the only
upper-triangular element in its orbit is $(2, 0, 0, 0)\in \YDmu_{4, 3}$.

\section{Kuznetsov--Polishchuk Construction}
\subsection{Equivariant vector bundles}
Let $\bfG$ be a simple algebraic group, and let $\bfP\subset \bfG$
be a maximal parabolic subgroup. Such subgroups are in one to one correspondence
with simple roots of $\bfG$. We are interested in $D(X)$, where $X=\bfG/\bfP$.
Since we know that all exceptional objects in $D(X)$ admit an equivariant structure,
we may being with studying the equivariant derived category $D^\bfG(X)$.
It is well known that the category of $\bfG$-equivariant sheaves
$\Coh^\bfG(X)$ is equivalent to the category of representations
\[
   \Coh^\bfG(X) \cong \Rep(\bfP),
\]
see~\cite{Bondal1990b}. Under this equivalence irreducible equivariant
vector bundles correspond to irreducible representations of $\bfP$, which
is not a reductive group. Denote by $\bfU\subset \bfP$ the unipotent radical,
and by $\bfL=\bfP/\bfU$ the semisimple Levi quotient. The semisimple representations
of $\bfP$ are precisely those restricted from $\bfL$.

Denote by $P^+_\bfL$ the dominant weight lattice of $\bfL$, and by
$\CU^\lambda$ the $\bfG$-equivariant vector bundle on $X$ corresponding
to $\lambda\in P^+_\bfL$. It is rather easy to construct an infinite
exceptional collection in $D^\bfG(X)$.

\begin{theorem}[{\cite[Theorem~3.4]{KuznetsovPolishchuk2016}}]\label{thm:gpg}
   Let $\xi$ denote the fundamental weight corresponding to the simple root
   associated with the maximal parabolic $\bfP\subset\bfG$.
   Consider the partial order on $P^+_\bfL$ \emph{opposite} to the so called
   \emph{$\xi$-ordering} given by the formula
   \[
      \lambda \preceq \mu \quad \text{if } (\xi, \lambda) < (\xi, \mu) \text{ or } \lambda=\mu.
   \]
   The bundles $\{\CU^\lambda \mid \lambda\in \bfP\subset\bfG\}$ form
   an infinite full exceptional collection in $D^\bfG(X)$.
\end{theorem}

There is an obvious forgetful functor
\[
   {\sf Fg}: D^\bfG(X)\to D(X).
\]
The idea of Kuznetsov and Polishchuk was to find a condition
under which a subcollection of the collection from Theorem~\ref{thm:gpg}
is mapped to an exceptional collection in the non-equivariant derived category.
They found one and called it the \emph{exceptional block} condition.

\subsection{Exceptional blocks}
Since the group $\bfG$ is reductive, the functor of taking $\bfG$-invariants
is exact. Thus, for any equivariant objects
$F, G\in D^\bfG(X)$ one has $\Ext_\bfG^i(F, G) = (\Ext_{D(X)}^i(F, G))^\bfG$,
and there is little chance that even exceptional objects are sent to exceptional
objects via the forgetful functor.

For any $\lambda, \mu\in P^+_\bfL$ the forgetful functor induces a linear map
\[
   {\sf Fg}: \Ext_\bfG^i(\CU^\lambda, \CU^\mu) \to \Ext_{D(X)}^i(\CU^\lambda, \CU^\mu),
\]
and for any triple $\lambda, \mu, \nu \in P^+_\bfL$ one gets a bilinear composition map
\[
   \Hom(\CU^\nu, \CU^\mu)\otimes_\kk \Ext_\bfG^i(\CU^\lambda, \CU^\nu)\to \Ext_{D(X)}^i(\CU^\lambda, \CU^\mu).
\]

\begin{definition}[{\cite[Definition~3.1]{KuznetsovPolishchuk2016}}]\label{def:reb}
   A subset $B\subset P^+_\bfL$ of dominant weights is called a \emph{(right) exceptional block} if the natural morphism
   \[
      \bigoplus_{\nu \in B}\Hom(\CU^\nu, \CU^\mu)\otimes_\kk \Ext_\bfG^i(\CU^\lambda, \CU^\nu)\to \Ext_{D(X)}^i(\CU^\lambda, \CU^\mu)
   \]
   is an isomorphism for all $\lambda, \mu\in B$.
\end{definition}

The definition of a right exceptional block says that any non-equivariant extension
between two irreducible equivariant bundles associated with two elements
of the block can be decomposed in a unique way as a sum of equivariant extensions
from the target into equivariant bundles associated with (other) elements of the block composed
with non-equivariant homomorphisms into the target.

The adjective \emph{right} does not appear in~\cite{KuznetsovPolishchuk2016};
however, it turns out that it is often convenient to consider left exceptional blocks.

\begin{definition}\label{def:leb}
   A subset $B\subset P^+_\bfL$ of dominant weights is called a \emph{left exceptional block} if the natural morphism
   \[
      \bigoplus_{\nu \in B} \Ext_\bfG^i(\CU^\nu, \CU^\mu)\otimes_\kk \Hom(\CU^\lambda, \CU^\nu)\to \Ext_{D(X)}^i(\CU^\lambda, \CU^\mu)
   \]
   is an isomorphism for all $\lambda, \mu\in B$.
\end{definition}

Here, we ask for each non-equivariant extension to decompose into a sum of
non-equivariant homomorphisms
followed by equivariant extensions.

Let $B\subset P^+_\bfL$ be an right exceptional block. According to Theorem~\ref{thm:gpg},
the bundles $\{\CU^\lambda \mid \lambda \in B\}$ form a graded exceptional collection in $D^\bfG(X)$.
Let $\{\CE^\lambda_B \mid \lambda \in B\}$ denote the \emph{graded right dual}
exceptional collection.

\begin{proposition}[{\cite[Proposition 3.9]{KuznetsovPolishchuk2016}}]\label{prop:reb}
   Let $B\subset P^+_\bfL$ be a right exceptional block. The objects
   \[\{{\sf Fg}(\CE^\lambda_B) \mid \lambda \in B\}\] form a graded exceptional
   collection in $D(X)$ with respect to the $\xi$-ordering.
\end{proposition}

There is a similar result for left exceptional blocks: for a left exceptional
block $B\subset P^+_\bfL$ let $\{\CF^\lambda_B \mid \lambda \in B\}$ denote the \emph{graded left dual}
exceptional collection to $\{\CU^\lambda \mid \lambda \in B\}$.

\begin{proposition}\label{prop:leb}
   Let $B\subset P^+_\bfL$ be a left exceptional block. The objects
   \[\{{\sf Fg}(\CF^\lambda_B) \mid \lambda \in B\}\] form a graded exceptional
   collection in $D(X)$ with respect to the $\xi$-ordering.
\end{proposition}

The proof of Proposition~\ref{prop:reb} is quite easy. The idea of the exceptional
block condition is beautiful and very powerful.
Most of~\cite{KuznetsovPolishchuk2016} is dedicated to constructing enough
orthogonal exceptional blocks for isotropic Grassmannians (quotients by maximal
parabolics of the simple groups of types $B$, $C$, and $D$) to produce
exceptional collections of maximal possible length (equal to the rank of $K_0$).
These results will be covered below.
The authors also formulated a certain curious conjecture
for classical Grassmannians, which we discuss in the following two sections.

\subsection{Generalized staircase complexes}
We have seen that the bundles $\Sigma^\lambda\CU^*$ (dually, $\Sigma^\lambda\CU$)
are exceptional on $\Gr(k, V)$ for $\lambda\in\YD_{k, n-k}$, and so are
$\Sigma^\mu\CU^\perp$ (dually $\Sigma^\mu(V/\CU)$) for $\mu\in\YD_{n-k, k}$.
It turns out, there is a way to ``interpolate'' between the collections
$\langle \Sigma^\lambda\CU^* \mid \lambda\in \YD_{k, n-k} \rangle$ and
$\langle \Sigma^\lambda(V/\CU)\mid \mu\in \YD_{n-k, k}\rangle$.
In the process, certain interesting exceptional bundles naturally appear.
In the present section we give a geometric description of these bundles
and present a generalization of staircase complexes.

To a pair of diagrams $\lambda\in \YD_{k, n-k}$ and $\mu\in \YD_{n-k, k}$ such that
\begin{equation}\label{eq:block-ineq}
   \ydw(\lambda) + \ydh(\mu) \leq n-k \qquad \text{and} \qquad
   \ydh(\lambda) + \ydw(\mu) \leq k
\end{equation}
we associate an equivariant exceptional vector bundle $\CF^{\lambda, \mu}$.
Its dual will be denoted by $\CE^{\lambda, \mu} = \left(\CF^{\lambda, \mu}\right)^*$.

Let $h= \ydh(\lambda)$ and assume that $0<h<k$.
Consider the diagram
\begin{equation}\label{eq:gr-elm1}
   \begin{tikzcd}
      & \Fl(k-h, k; V) \arrow[dr, "q"] \arrow[dl, "p"'] & \\
      \Gr(k, V) & & \Gr(k-h, V).
   \end{tikzcd}
\end{equation}
We denote by $\CU$ and $\CW$ the tautological bundles on
$\Gr(k, V)$ and $\Gr(k-h, V)$, respectively, as well as their
pullbacks under $p$ and $q$. Thus, the universal flag on
$\Fl(k-h, k; V)$ is $\CW\subset\CU\subset V$.

\begin{proposition}[\cite{Fonarev2014}]\label{prop:gr-flm}
   If $\lambda$ and $\mu$ satisfy~\eqref{eq:block-ineq} and $h=\ydh(\lambda)$
   satisfies $0 < h < k$, then
   \begin{equation*}
      R^ip_*\!\left(\Sigma^\lambda(\CU/\CW)\otimes q^*\Sigma^\mu\CW^\perp\right) = 0
      \quad \text{for all}\quad i > 0
   \end{equation*}
   and
   \begin{equation}\label{eq:flm}
      \CF^{\lambda, \mu} = p_*\!\left(\Sigma^\lambda(\CU/\CW)\otimes q^*\Sigma^\mu\CW^\perp\right)
   \end{equation}
   is an exceptional vector bundle on $\Gr(k, V)$.
\end{proposition}

Since we introduced the extra condition on $h$, we extend the definition
of $\CF^{\lambda, \mu}$ by putting
\[
   \CF^{\lambda, 0} = \Sigma^\lambda\CU \qquad \text{and} \qquad
   \CF^{0, \mu} = \Sigma^\mu\CU^\perp.
\]
Thus, the bundles $\CF^{\lambda,\mu}$ become a generalization
of the bundles $\Sigma^\lambda\CU$ and $\Sigma^\mu\CU^\perp$.

There is an alternative geometric construction of $\CF^{\lambda, \mu}$,
which can be obtained via the usual isomorphism $\Gr(k, V)\simeq \Gr(n-k, V)$.
Assume that $w=\ydw(\lambda)$ is such that $0 < w < n-k$. Consider
the diagram
\begin{equation}\label{eq:gr-elm2}
   \begin{tikzcd}
      & \Fl(k, n-w; V) \arrow[dr, "g"] \arrow[dl, "f"'] & \\
      \Gr(k, V) & & \Gr(n-w, V).
   \end{tikzcd}
\end{equation}
Denote by $\CU$ and $\CK$ the tautological bundles on
$\Gr(k, V)$ and $\Gr(n-w, V)$, respectively, as well as their
pullbacks under $f$ and $g$. Thus, the universal flag on
$\Fl(k, n-w; V)$ is $\CU\subset\CK\subset V$.

\begin{proposition}[\cite{Fonarev2014}]\label{prop:gr-flm2}
   If $\lambda$ and $\mu$ satisfy~\eqref{eq:block-ineq} and $w=\ydh(\lambda)$
   satisfies $0 < w < n-k$, then
   \begin{equation*}
      R^if_*\!\left(\Sigma^\mu(\CK/\CU)^*\otimes g^*\Sigma^\lambda\CK\right) = 0
      \quad \text{for all}\quad i > 0
   \end{equation*}
   and
   \begin{equation}\label{eq:flm2}
      f_*\!\left(\Sigma^\mu(\CK/\CU)^*\otimes g^*\Sigma^\lambda\CK\right) \simeq \CF^{\lambda, \mu}.
   \end{equation}
\end{proposition}

\begin{example}
   The first nontrivial example of a diagram of the form $\CF^{\lambda, \mu}$
   is given by the universal extension
   \begin{equation}
      0 \to \CU\otimes\CU^\perp \to \CF^{\Box, \Box} \to \CO \to 0,
   \end{equation}
   the dual of which is 
   \begin{equation}
      0 \to \CO \to \CE^{\Box, \Box}\to \CU^*\otimes(V/\CU) \to 0.
   \end{equation}
\end{example}

In the following it will be convenient for us to work with the dual
bundles $\CE^{\lambda, \mu}$. Let us list some of their properties.
First, as we have already seen, if $\mu$ is empty,
then $\CE^{\lambda, \mu} = \CE^{\lambda, 0}\simeq \Sigma^\lambda\CU^*$,
and if $\lambda$ is empty, then $\CE^{0, \mu}=\Sigma^\mu(V/\CU)$.
Second,
\begin{equation*}
   \CE^{\lambda, \mu} \in \langle \Sigma^\alpha\CU^*\otimes \Sigma^\beta (V/\CU)
   \mid \alpha \subset \lambda,\ \beta\subset \mu \rangle,
\end{equation*}
and there is actually an epimorphism $\CE^{\lambda, \mu} \to \Sigma^\lambda\CU^*\otimes \Sigma^\mu (V/\CU)$
whose kernel belongs to the subcategory $\langle \Sigma^\alpha\CU^*\otimes \Sigma^\beta (V/\CU)
\mid \alpha \subsetneq \lambda,\ \beta\subsetneq \mu \rangle$.
Third, in Proposition~\ref{prop:gr-flm} one could choose any integer
$h$ as long as $\ydh(\lambda)\leq h$ and $\ydw(\mu) \leq k-h$.
Similarly, in Proposition~\ref{prop:gr-flm2} one could choose
any integer $w$ as long as $\ydw(\lambda)\leq w$ and $\ydh(\mu)\leq n-k-w$.

Let $\lambda, \mu$ be as above, and let $w=\ydw(\lambda)>0$.
Recall that $\bar\lambda$ denotes the diagram obtained from $\lambda$ by removing the first
row. Denote by $\mu(1)$ the diagram obtained from $\mu$ by adding
a full column of height $n-k-w$. In the following $\lambda^{(i)}$
are precisely the diagrams introduced in~\eqref{eq:circ}.

\begin{proposition}\label{prop:gr-genstc}
   Let $w=\ydw(\lambda)$.
   There is an exact sequence of vector bundles on $\Gr(k, V)$ of the form
   \begin{equation}\label{eq:gr-genstc}
      0\to \CE^{\bar\lambda, \mu(1)}(-1)\to \Lambda^{b_\lambda^{(w)}}V^*\otimes \CE^{\lambda^{(w)}, \mu} \to
      \cdots
      \to \Lambda^{b_\lambda^{(1)}}V^*\otimes \CE^{\lambda^{(1)}, \mu} \to
      \CE^{\lambda, \mu} \to 0,
   \end{equation}
   where $b_\lambda^{(i)} = |\lambda/\lambda^{(i)}| = |\lambda|-|\lambda^{i}|$. 
\end{proposition}

Using the isomorphism between $\Gr(k, V)$ and $\Gr(n-k, V^*)$,
one gets the following dual version of Proposition~\ref{prop:gr-genstc}.

\begin{proposition}\label{prop:gr-genstc2}
   Let $w'=\ydw(\mu)$.
   There is an exact sequence of vector bundles on $\Gr(k, V)$ of the form
   \begin{equation}\label{eq:gr-genstc2}
      0\to \CE^{\lambda(1), \bar\mu}(-1)\to \Lambda^{b_\lambda^{(w')}}V^*\otimes \CE^{\lambda, \mu^{(w)}} \to
      \cdots
      \to \Lambda^{b_\lambda^{(1)}}V^*\otimes \CE^{\lambda, \mu^{(1)}} \to
      \CE^{\lambda, \mu} \to 0,
   \end{equation}
   where $b_\mu^{(i)} = |\lambda/\lambda^{(i)}| = |\lambda|-|\lambda^{i}|$,
   and $\lambda(1)$ is obtained from $\lambda$ by adding a full column
   of height $k-w'$.
\end{proposition}

\begin{definition}
   Long exact sequences~\eqref{eq:gr-genstc} and~\eqref{eq:gr-genstc2}
   are called \emph{generalized staircase complexes}.
\end{definition}

\begin{remark}
   The proofs of Propositions~\ref{prop:gr-genstc} and~\ref{prop:gr-genstc2}
   appear in~\cite{FonarevKP2015} (see Theorems~4.3 and~4.4, respectively);
   however, what we denote by $\CF^{\lambda, \mu}$
   in the present survey is denoted by $\CE^{\lambda, \mu}$ there.
   We present our apology to the reader for the change of notation,
   which we did to make it in line with the definitions of right
   and left exceptional blocks.
\end{remark}

Long exact sequences~\eqref{eq:gr-genstc} and~\eqref{eq:gr-genstc2}
are obvious generalizations of~\eqref{eq:staricase}, and are interesting
even when $\lambda$ or $\mu$ is the zero diagram.
For instance, when $\lambda = (w)$ for some $0 < w < n-k$ and $\mu$ is empty,
one has $\CE^{\lambda, \mu}\simeq S^w\CU^*$.
In this case $\bar\lambda$ is empty and $\mu(1) = (n-k-w)^T$,
so $\CE^{\bar\lambda, \mu(1)}(1)\simeq \Lambda^{n-k-w}(V/\CU)(1)\simeq \Lambda^w\CU^\perp$,
and the sequence~\eqref{eq:gr-genstc} is the well-known
long exact sequence
\[
0\to \Lambda^w\CU^\perp \to \Lambda^wV^*\otimes\CO \to \Lambda^{w-1}V^*\otimes\CU^*
\to \cdots \to
V^*\otimes S^{w-1}\CU^*\to S^w\CU^*\to 0
\]
obtained from the dual tautological short exact sequence $0\to \CU^\perp\to V^*\to \CU^*\to 0$.

\subsection{Kuznetsov--Polishchuk collections for classical Grassmannians}\label{ssec:kp-gr}
The first application of generalized staircase complexes that we encounter
is related with the following story which was conjectured in~\cite[Section~9.6]{KuznetsovPolishchuk2016}.
The proof of the statements was originally announced in~\cite{Fonarev2014} and presented
in full in~\cite{FonarevKP2015}.

Let $p=(w,h)\in\RR^2$ be a point in our favorite rectangle of width
$n-k$ and height $k$; that is, we assume that $0\leq w\leq n-k$ and $0\leq h\leq k$.
To such a point we want to associate a collection of weights
of the maximal parabolic $\rmP\subset \GL(V)$. Since we do not want to introduce
the necessary notation, we simply describe this collection via
the corresponding irreducible equivariant vector bundles:
these are of the form $\Sigma^\lambda\CU^*\otimes\Sigma^\mu(V/\CU)$,
with $\lambda\in\YD_{n-k-w, h}$ and $\mu\in\YD_{k-h, w}$
(where for arbitrary $a, b\in \RR$ we put $\YD_{a, b}=\YD_{\lfloor a \rfloor, \lfloor b \rfloor}$).
Remark that such a pair $(\lambda,\mu)$ satisfies conditions~\eqref{eq:block-ineq}.
We denote this set of weights by $B_p$, and by $\CB_p$ the subcategory
generated by the corresponding bundles,
\[
   \CB_p = \left\langle \Sigma^\lambda\CU^*\otimes\Sigma^\mu(V/\CU)
               \mid (\lambda,\mu)\in \YD_{n-k-w, h}\times\YD_{k-h, w} \right\rangle.
\]

\begin{proposition}[{\cite[Proposition~5.1]{FonarevKP2015}}]\label{prop:gr-kp}
   The subcategory $\CB_p$ is generated by a full exceptional collection
   \[
      \CB_p = \left\langle \CE^{\lambda,\mu}
      \mid (\lambda,\mu)\in \YD_{n-k-w, h}\times\YD_{k-h, w} \right\rangle.
   \]
\end{proposition}

\begin{remark}
   It was originally conjectured in~\cite{KuznetsovPolishchuk2016}
   that the blocks $B_p$ are (right) exceptional blocks. Though it was
   stated in~\cite{Fonarev2014} that this is indeed the case, the proof
   of Proposition~\ref{prop:gr-kp} does not rely on this fact.
   What we do get though is that the objects $\CE^{\lambda,\mu}$ form
   the right dual exceptional collection to $\Sigma^\lambda\CU^*\otimes\Sigma^\mu(V/\CU)$
   in the equivariant derived category.
\end{remark}

Consider a monotonous path $\gamma$ in $\RR^2$ going from $(0, 0)$ to $(n-k, k)$.
Denote by $p_0, p_1, p_2, \ldots, p_l$ the~points
of intersection of $\gamma$ with the grid $\ZZ\times\RR \cup \RR\times\ZZ$,
listed from left to right. In particular, one has $p_0=(0, 0)$ and
$p_l=(n-k, k)$. The following theorem was stated originally
as a conjecture in~\cite{KuznetsovPolishchuk2016} (see Conjecture~9.8).

\begin{theorem}[\cite{FonarevKP2015}]\label{thm:gr-kp}
   For any $\gamma$ as above, there is a semiorthogonal decomposition
   \begin{equation}\label{eq:gr-kp}
      D^b(\Gr(k, V)) = \langle \CB_{p_0}, \CB_{p_1}(1), \ldots, \CB_{p_l}(l) \rangle.
   \end{equation}
\end{theorem}

Generalized staircase complexes were specifically constructed to prove
fullness of the decompositions~\eqref{eq:gr-kp}. Namely, via a combinatorial
argument one can reduce to the case when $\gamma$ is of very specific
form (goes very closely along the left border of the rectangle, then
follows the top border), and it is not hard to see that the corresponding
decompositions comes from Kapranov's exceptional collection.

\begin{remark}
   When $n=2k$, one can choose $\gamma$ to be diagonal in the rectangle $[0,k]\times[0,k]$.
   The corresponding decomposition has some very nice symmetry.
\end{remark}

\section{Symplectic Grassmannians}
\subsection{Symplectic Grassmannians}
We are ready to discuss isotropic Grassmannians associated with the simple
group $\bfG=\SP(V)$ of type $C_n$, where $V$ is a $2n$-dimensional
vector spaces equipped with a non-degenerate skew-symmetric bilinear form
$\omega\in\Lambda^2V^*$. There are precisely $n$ maximal parabolic
subgroups $\bfP_i\subset\bfG$.
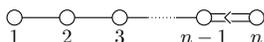
\begin{figure}[H]
   \begin{tikzpicture}
      \dynkin[labels={1,2,3, n-1, n}, scale=2, Bourbaki arrow] C{ooo...oo}
   \end{tikzpicture}
   \caption{Dynkin diagram $C_n$.}
   \label{fig:cn}
\end{figure}
We follow the Bourbaki numbering
of simple roots, which is displayed on Figure~\ref{fig:cn}. The quotient
$\bfG/\bfP_k \simeq \IGr(k, V)$ can be identified with the Grassmannian
of subspaces $U\subset V$ isotropic with respect to~$\omega$.
In particular, $\IGr(k, V)$ embeds into $\Gr(k, V)$, and can actually
be described as the zero locus of the regular section of $\Lambda^2\CU^*$
on $\Gr(k, V)$ which corresponds to $\omega$
(recall that $\Gamma(\Gr(k, V), \Lambda^2\CU^*)\simeq \Lambda^2V^*$).
The~isotropic Grassmannian $\IGr(k, V)$ has Fano index $2n-k+1$,
and $\rk K_0(\IGr(k, V)) = \binom{n}{k}2^n$.
We denote by $\CU$ the universal bundle (restricted from $\Gr(k, V)$)
on $\IGr(k, V)$.
The form $\omega$ induces an isomorphism $V\to V^*$, and one
gets a commutative diagram
\[
   \begin{tikzcd}
      0 \arrow[r] & \CU \arrow[r] \arrow[d, hook] & V \arrow[r] \arrow[d, "\omega"'] & V/\CU \arrow[r] \arrow[d, two heads] & 0 \\
      0 \arrow[r] & \CU^\perp \arrow[r] & V^* \arrow[r] & \CU^* \arrow[r] & 0
   \end{tikzcd}.
\]
The quotient $\CU^\perp\!/\CU$ is naturally a symplectic vector bundle.
By looking at the Levi subgroup of $\bfP_i$, we see that
the irreducible equivariant vector bundles on $\IGr(k, V)$ are all
of the form
\[
   \Sigma^\lambda\CU^*\otimes (\CU^\perp\!/\CU)^{\langle \mu \rangle},
\]
where $\lambda$ is a dominant weight of $\GL_k$, $\mu$ is a dominant weight
of $\SP_{n-k}$, and $(-)^{\langle\mu\rangle}$ denotes the symplectic
Schur functor.

\subsection{Symplectic Grassmannians of planes}
The first isotropic symplectic Grassmannian $\SP(V)/\bfP_1$ is not very
interesting. Indeed, $\IGr(1, V)\simeq \PP(V)$, and this case has been
covered by Beilinson. Thus, we turn to the case $k=2$.
Recall that $\IGr(k, V)$ is the zero locus of a regular section of
$\Lambda^2\CU^*$ on $\Gr(k, V)$. Thus, $\IGr(2, V)$ is just a smooth
hyperplane section of $\Gr(2, V)$ with respect to $\CO(1)\simeq \Lambda^2\CU^*$.
A full exceptional collection in $D(\IGr(2, V))$ was constructed
by A.~Kuznetsov in~\cite{Kuznetsov2008} using the following simple observation
which lies in the core of his theory of Homological Projective Duality.

\begin{proposition}
   Let $X$ be a smooth projective variety with a very ample line bundle $\CO(1)$,
   and let $Y$ be a smooth hyperplane section of $X$ with respect to $\CO(1)$.
   Denote by $\iota:Y \to X$ the embedding morphism.
   Assume that $D(X)$ is equipped with a Lefschetz decomposition
   \[
      D(X) = \langle \CA_0, \CA_1(1), \ldots, \CA_{m-1}(m-1) \rangle.
   \]
   Then for all $1\leq i < m$ the pullback functor induces a fully faithful
   embedding
   \[
      \iota^* : \CA_i(i)\hookrightarrow D(Y),
   \]
   and there is a semiorthogonal decomposition of the form
   \begin{equation}\label{eq:hpd}
      D(Y) = \langle \CA_Y, \iota^*(\CA_1(1)), \ldots, \iota^*(\CA_{m-1}(m-1)) \rangle.      
   \end{equation}
\end{proposition}

In other words, if one has a Lefschetz decomposition with a rather small initial
block, then one gets quite a lot of information about the derived categories
of its hyperplane sections: the subcategory generated by the remaining blocks
embeds fully and faithfully into them.
The best possible scenario happens when $\CA_Y$ in~\eqref{eq:hpd}
vanishes. It turns out, this is the case for $\IGr(2, V)$.

\begin{theorem}[{\cite[Theorem~5.1]{Kuznetsov2008}}]\label{thm:igr2}
   There is a Lefschetz decomposition of $D(\IGr(2, V))$ with the basis
   \[
   (\CO, \CU^*, \ldots, S^{n-2}\CU^*, S^{n-1}\CU^*), \qquad o = (2n-1, 2n-1, \ldots, 2n-1, n-1).
   \]
\end{theorem}

\begin{remark}
   The exceptional collection produced by Theorem~\ref{thm:igr2} is, up to a twist by $\CO(-1)$,
   is precisely the collection given in Remark~\ref{rmk:gr2} with the first column dropped.
   Namely, we get an exceptional collection of the form
   \begin{equation*}
      \begin{pmatrix}
         S^{n-1}\CU^* & S^{n-1}\CU^*(1) & \cdots & S^{n-1}(n-2) &  \\
         S^{n-2}\CU^* & S^{n-2}\CU^*(1) & \cdots & S^{n-2}(n-2) & \cdots & S^{n-2}\CU^*(2n-2) \\
         \vdots & \vdots &  & \vdots & & \vdots \\
         \CU^* & \CU^*(1) & \cdots & \CU(n-1) & \cdots & \CU^*(2n-2) \\
         \CO^* & \CO^*(1) & \cdots & \CO(n-1) & \cdots & \CO^*(2n-2)
      \end{pmatrix}.
   \end{equation*}
\end{remark}

\subsection{Kuznetsov--Polishchuk collections}
Kuznetsov and Polishchuk constructed exceptional blocks that
produce exceptional collections of maximal length in $\IGr(k, V)$
for all $k$. The following results are formulated in~\cite[Theorem~9.2]{KuznetsovPolishchuk2016}.
We will also say that a collection of irreducible equivariant vector bundles forms
an exceptional block if the corresponding dominant weights do.

\begin{proposition}\label{prop:igr-b1}
   Let $0\leq t \leq k-1$. Then the bundles
   \[
      \{ \Sigma^\lambda\CU^*\otimes (\CU^\perp\!/\CU)^{\langle \mu \rangle}
      \mid \lambda\in\YD_{t, 2n-k-t},\ \mu\in\YD_{n-k,\lfloor(k-t/2)\rfloor} \}
   \]
   form an exceptional block. We denote by $\CA_t$ the corresponding
   admissible subcategory in $D(\IGr(k, V))$.
\end{proposition}

\begin{proposition}\label{prop:igr-b2}
   Let $k \leq t \leq 2n-k$. Then the bundles
   \[
      \{ \Sigma^\lambda\CU^*
      \mid \lambda\in\YD_{k-1, 2n-k-t} \}
   \]
   form an exceptional block. We denote by $\CA_t$ the corresponding
   admissible subcategory in $D(\IGr(k, V))$.
\end{proposition}

\begin{theorem}\label{thm:igr-kp}
   The subcategories $\CA_0,\CA_1(1),\ldots,\CA_{2n-k}(2n-k)$
   are semiorthogonal in $D(\IGr(k, V))$. Since each of them is generated
   by an exceptional collection, there is an exceptional collection
   in $D(\IGr(k, V))$ of length $2^k\binom{n}{k}$, which equals
   $\rk K_0(\IGr(k, V))$.
\end{theorem}

\begin{remark} Let us say a few words about Theorem~\ref{thm:igr-kp}.
   \begin{enumerate}
      \item The bundles in Proposition~\ref{prop:igr-b2} are actually fully orthogonal
      in the equivariant derived category. Thus, the graded right dual collection
      coincides with the original one. One can check directly that for
      $k\leq t\leq 2n-k$ the bundles
      \[
         \langle \Sigma^\lambda\CU^* \mid \lambda\in\YD_{k-1, 2n-k-t} \rangle
      \]
      form an exceptional collection in $D(\IGr(k, V))$.

      \item The blocks 
      \[
         \langle \CA_k(k), \CA_{k+1}(k+1), \ldots, \CA_{2n-k}(2n-k)\rangle
      \]
      can be neatly packed into a single subcategory generated by an exceptional
      collection
      \[
         \langle \Sigma^{\lambda}\CU^*(k) \mid \lambda \in \YD_{k, 2n-2k} \rangle.
      \]
   \end{enumerate}
\end{remark}

\begin{conjecture}[{\cite[Conjecture~1.6]{KuznetsovPolishchuk2016}}]\label{conj:kp-igr}
   The exceptional collections in Theorem~\ref{thm:igr-kp} are full.
\end{conjecture}

So far, Conjecture~\ref{conj:kp-igr} has been proved only for Lagrangian
Grassmannians in~\cite{Fonarev2022}. We discuss this result below.

\subsection{Lagrangian Grassmannians and Lagrangian staircase complexes}
Consider the Lagrangian Grassmannian $\LGr(V) = \IGr(n, V)$.
Theorem~\ref{thm:igr-kp}
produces a collection of subcategories
\[
   \CA_t = \langle \CE^\lambda_t \mid \lambda \subset \YD_{t, k-t} \rangle,
\]
where the objects $\CE^\lambda_t$ form the graded right dual exceptional
collection to $\langle \Sigma^\lambda\CU^* \mid \lambda \subset \YD_{t, k-t} \rangle$
in the equivariant derived category.

We first observe that the subscript $t$ can be dropped from $\CE^\lambda_t$
since the corresponding object does not depend on the block.
Indeed, for any $\lambda$ such that $\ydh(\lambda)+\ydw(\lambda)\leq n+1$ we are going
to construct an exceptional equivariant vector bundle on $\LGr(V)$,
which we denote by $\CE^\lambda$. This bundle will coincide
with $\CE^\lambda_t$ if $\lambda \in \YD_{t, k-t}$.
\begin{remark}
   We actually construct more exceptional bundles than what Theorem~\ref{thm:igr-kp}
   gives us: our condition is that $h+w=n+1$, while Theorem~\ref{thm:igr-kp}
   is concerned with diagrams $\lambda$ such that $\ydh(\lambda)+\ydw(\lambda)\leq n$.
\end{remark}

There are two constructions of $\CE^\lambda$ presented in~\cite{Fonarev2022}.
In both cases we are concerned with the dual objects, which we denote by
$\CF^\lambda$ (a similar thing happened in section~\ref{ssec:kp-gr}).
Let $h, w\geq 1$ be such that $h+w=n+1$.
Consider the diagram
\begin{equation}\label{eq:lgr-el1}
   \begin{tikzcd}
      & \IFl(w, n; V) \arrow[dr, "g"] \arrow[dl, "f"'] & \\
      \LGr(V) & & \IGr(w, V),
   \end{tikzcd}
\end{equation}
and denote by $\CW\subset\CU$ the universal flag on $\IFl(w, n; V)$.

\begin{proposition}[{\cite[Proposition~3.1]{Fonarev2022}}]\label{prop:lgr-fl1}
   Let $\lambda\in\YD_{h-1, w}$. The object
   \[
      \CF^\lambda = f_*g^* (\CW^\perp\!/\CW)^{\langle \lambda \rangle}
   \]
   is an exceptional vector bundle on $\LGr(V)$ which depends only on $\lambda$
   and not on $w$. 
\end{proposition}

If $w>1$, we also consider the diagram
\begin{equation}\label{eq:lgr-el2}
   \begin{tikzcd}
      & \IFl(w-1, n; V) \arrow[dr, "\tilde{g}"] \arrow[dl, "\tilde{f}"'] & \\
      \LGr(V) & & \IGr(w-q, V),
   \end{tikzcd}
\end{equation}
and denote by $\CH\subset\CU$ the universal flag on $\IFl(w-1, n; V)$.

\begin{proposition}[{\cite[Proposition~3.2]{Fonarev2022}}]\label{prop:lgr-fl2}
   Assume $w\geq 2$. Let $\lambda\in\YD_{h, w}$ be such that $\lambda_h>0$. The object
   \[
      \CF^\lambda = \tilde{f}_*\left( \det (\CU/\CH)\otimes 
       \tilde{g}^* (\CH^\perp\!/\CH)^{\langle \lambda(-1) \rangle}
       \right),
   \]
   where $\lambda(-1)=(\lambda_1-1, \ldots, \lambda_h-1)$,
   is an exceptional vector bundle on $\LGr(V)$.
   If $\lambda \in \YD_{h'-1, w'}$ for some $h'+w'=n+1$, then
   $\CF^\lambda$ coincides with the one defined in Proposition~\ref{prop:lgr-fl1}.
\end{proposition}

As one can see, Propositions~\ref{prop:lgr-fl1} and~\ref{prop:lgr-fl2}
cover all $\lambda$ with $\ydh(\lambda)+\ydw(\lambda)\leq n+1$, but
not in a single take. The second construction does not have this disadvantage.
We consider diagram~\eqref{eq:lgr-el2} once again. First, observe the following.

\begin{lemma}[{\cite[Lemma~3.4]{Fonarev2022}}]\label{lm:lgr-ul}
   There is an exceptional collection in $D(\IGr(w, V))$ of the form
   \begin{equation}\label{eq:lgr-ul}
      \langle \Sigma^\mu \CU^* \mid \mu\in \YD_{w, h}\rangle.
   \end{equation}
\end{lemma}

Denote the left graded dual exceptional collection to~\eqref{eq:lgr-ul} by
\begin{equation*}
   \langle \CG^\lambda \mid \lambda \in \YD_{h, w} \rangle.
\end{equation*}
The defining property of this collection is the following:
\begin{equation*}
   \CG^\lambda \in \langle \Sigma^\mu \CU^* \mid \mu\in \YD_{w, h}\rangle,
   \quad
   \Ext^\bullet(\Sigma^\mu\CU^*, \CG^\lambda) = \begin{cases}
      \kk[-|\lambda|] & \text{if } \lambda = \mu^T, \\
      0 & \text{otherwise}.
   \end{cases}
\end{equation*}

\begin{proposition}[{\cite[Proposition~3.6]{Fonarev2022}}]\label{prop:lgr-gl}
   There is an isomorphism
   \[
      \CF^\lambda \simeq f_*g^* \CG^\lambda.
   \]
\end{proposition}

\begin{example}
   Let us give some examples of $\CE^\lambda\simeq (\CF^\lambda)^*$.
   \begin{enumerate}
      \item If $\lambda = (t)$ for some $2\leq t \leq n$, then $\CE^{(t)}$
      is an extension of the form
      \[
         0\to S^{t-2}\CU^* \to \CE^{(t)} \to S^t\CU^* \to 0.
      \]
      \item If $\lambda = (t)^T$ for some $1\leq t \leq n$, then
      \[
         \CE^{(t)^T} \simeq \Lambda^t\CU^*.
      \]
      \item For $\lambda = (3, 1)$ there is an extension
      \[
         0\to S^2\CU^* \oplus \Lambda^2\CU^* \to \CE^{(3, 1)} \to \Sigma^{(3, 1)}\CU^* \to 0.
      \]
   
   \end{enumerate}
\end{example}

Descriptions of $\CF^\lambda$ (thus, of $\CE^\lambda$ which are simply the duals)
given by Propositions~\ref{prop:lgr-fl1}, \ref{prop:lgr-fl2}, and~\ref{prop:lgr-gl}
allow to construct what we call \emph{Lagrangian staircase complexes}.

Recall that the set $\YD_{h, w}$ for $w+h=n+1$ can be identified with the
set of binary sequences of length $n+1$ with exactly $w$ terms equal to $0$.
Thus, the disjoint union $\sqcup_{w+w=n+1}\YD_{h, w}$ is in bijection
with the set of all binary sequences of length $n+1$, which is $\{0,1\}^{n+1}$.
Fix a generator $g\in \ZZ/(2n+2)\ZZ$ and consider the following action
of the group $\ZZ/(2n+2)\ZZ$ on $\{0,1\}^{n+1}$:
\begin{equation*}
   g: a_0a_1\ldots a_{n} \mapsto (a_n-1)a_0a_1\ldots a_{n-1}.
\end{equation*}
The action on the diagrams can be described as follows.
Let $\lambda=(\lambda_1,\ldots,\lambda_h)\in\YD_{w, h}$. Then
\[
   g: \lambda \mapsto \lambda' = \begin{cases}
      (\lambda_1,\ldots,\lambda_h, 0) & \text{if } \lambda_1 < w, \\
      (\lambda_2+1,\ldots,\lambda_h+1) & \text{if } \lambda_1 = w.
   \end{cases}
\]
In particular, $\lambda'\in \YD_{w-1, h+1}$ if $\lambda_1 < w$
and $\lambda'\in \YD_{w+1, h-1}$ if $\lambda_1 = w$.
We keep the same notation $\lambda'$ as in section~\ref{ssec:stair}
in order to show direct analogy between the classical and Lagrangian
cases. We strongly suggest the reader compares the following proposition
with Proposition~\ref{prop:starcase}.

\begin{proposition}{{\cite[Proposition~4.3]{Fonarev2022}}}\label{prop:lstair}
   Let $w, h>0$ be such that $w+h=n+1$, and let $\lambda\in\YD_{h, w}$
   be such that $\lambda_1 = w$. There is an exact sequence of vector bundles
   on $\LGr(V)$ of the form
   \begin{equation}\label{eq:lstair}
      0\to \CE^{\lambda'}(-1) \to V^{[b_\lambda^{(w)}]}\otimes \CE^{\lambda^{(w)}}
      \to \ldots \to V^{[b_\lambda^{(1)}]}\otimes \CE^{\lambda^{(1)}} \to
      \CE^\lambda \to 0,
   \end{equation}
   where $\lambda^{(i)}$ and $b_\lambda^{(i)}$ are the same as in
   Proposition~\ref{prop:starcase}, and $V^{[j]}$ denotes the $j$-th fundamental
   representation of $\SP(V)$.
   Complexes of the form~\eqref{eq:lstair} are called
   \emph{Lagrangian staircase complexes}.
\end{proposition}

Lagrangian staircase complexes were introduced precisely to prove
Conjecture~\ref{conj:kp-igr} for Lagrangian Grassmannians.

\begin{theorem}[{\cite[Theorem~4.6]{Fonarev2022}}]\label{thm:lgr}
   There is a semiorthogonal decomposition
   \begin{equation*}
      D(\LGr(V)) = \langle \CA_0,\, \CA_1(1),\, \CA_2(2),\, \ldots,\, \CA_{n}(n) \rangle,
   \end{equation*}
   where for each $0\leq h\leq n$ the subcategory $\CA_h$ is generated by an exceptional collection
   \begin{equation}\label{eq:lgr-block}
      \CA_h = \langle \CE^\lambda \mid \lambda \in \YD_{h, n-h} \rangle.
   \end{equation}
\end{theorem}

\begin{remark}
   For some strange reason we had to consider more exceptional objects
   than used in Theorem~\ref{thm:lgr} (the condition $h+w=n$ was extended
   to $h+w=n+1$). This strange doubling phenomenon is quite common
   when dealing with Lagrangian Grassmannians.
\end{remark}

Using Theorem~\ref{thm:sam} and the existence of full exceptional collections
on flag varieties of type $A$, one immediately deduces the following.

\begin{corollary}
   Let $1\leq i_1 < i_2 < \cdots < i_t < n$. Then the bounded derived category
   of $\IFl(i_1, \ldots, i_t, n;\, V)$ admits a full exceptional collection
   consisting of vector bundles.
\end{corollary}
\begin{proof}
   Apply Theorem~\ref{thm:sam}, the existence of full exceptional collections
   on flag varieties of type $A$, and the fact that the projection
   $\IFl(i_1, \ldots, i_t, n;\, V)\to \LGr(V)$ is isomorphic to the
   relative flag bundle $\Fl_{\LGr(V}(i_1, \ldots, i_t;\, \CU)$.
\end{proof}

There is an interesting duality result for collections~\eqref{eq:lgr-block}.

\begin{theorem}[{\cite[Theorem~3.1]{Fonarev2023}}]\label{thm:lgr-dual}
   Let $0\leq h \leq n$. The exceptional collection
   $\langle \CF^\mu \mid \mu\in\YD_{n-h, h} \rangle$ is the graded
   left dual exceptional collection to $\langle \CE^\lambda \mid \lambda \in \YD_{h, n-h} \rangle$.
   That is, for $\mu\in\in\YD_{n-h, h}$ one has
   \[
      \CF^\mu \in \langle \CE^\lambda \mid \lambda\in \YD_{h, n-h}\rangle,
      \quad
      \Ext^\bullet(\CF^\mu, \CE^\lambda) = \begin{cases}
         \kk[-|\lambda|] & \text{if } \lambda = \mu^T, \\
         0 & \text{otherwise}.
      \end{cases}
   \]
\end{theorem}

Finally, we note that
the full exceptional collections given by Theorem~\ref{thm:lgr} are in no reasonable
way Lefschetz. We end this section with minimal Lefschetz exceptional collections
for small $n$.

\begin{proposition}[\cite{Samokhin,SamokhinP,Fonarev2022}]\label{prop:lgr-lef}
   For $n=2, \ldots, 5$ there are minimal Lefschetz exceptional collections
   on $\LGr(n)$ of the following form.
   \begin{enumerate}
      \item For $n=2$ one has
      \[
         D(\LGr(2, 4)) = \begin{pmatrix*}
            \CU^* & & \\
            \CO & \CO(1) & \CO(2)
         \end{pmatrix*}.
      \]
      \item For $n=3$ one has
      \[
         D(\LGr(3, 6)) = \begin{pmatrix*}[r]
            \CU^* & \CU^*(1) & \CU^*(2) & \CU^*(3) \\
            \CO & \CO(1) & \CO(2) & \CO(3)
         \end{pmatrix*}.
      \]
      \item For $n=4$ one has
      \[
         D(\LGr(4, 8)) = \begin{pmatrix*}[r]
            \CE^{(2, 1)} & & & & \\
            \Lambda^2\CU^* & \Lambda^2\CU^*(1) & \Lambda^2\CU^*(2) & \Lambda^2\CU^*(3) & \Lambda^2\CU^*(4) \\
            \CU^* & \CU^*(1) & \CU^*(2) & \CU^*(3) & \CU^*(4) \\
            \CO & \CO(1) & \CO(2) & \CO(3) & \CO(4)
         \end{pmatrix*}.
      \]
      \item For $n=5$ one has
      \[
         D(\LGr(5, 10)) = \begin{pmatrix*}[r]
            \CE^{(2, 2)} & \CE^{(2, 2)}(1) & & & &  \\
            \CE^{(2, 1)} & \CE^{(2, 1)}(1) & \CE^{(2, 1)}(2) & \CE^{(2, 1)}(3) & \CE^{(2, 1)}(4) & \CE^{(2, 1)}(5) \\
            \CE^{(2)} & \CE^{(2)}(1) & \CE^{(2)}(2) & \CE^{(2)}(3) & \CE^{(2)}(4) & \CE^{(2)}(5) \\
            \Lambda^2\CU^* & \Lambda^2\CU^*(1) & \Lambda^2\CU^*(2) & \Lambda^2\CU^*(3) & \Lambda^2\CU^*(4) & \Lambda^2\CU^*(5) \\
            \CU^* & \CU^*(1) & \CU^*(2) & \CU^*(3) & \CU^*(4) & \CU^*(5) \\
            \CO & \CO(1) & \CO(2) & \CO(3) & \CO(4) & \CO(5)
         \end{pmatrix*}.
      \]
   \end{enumerate}
\end{proposition}

\subsection{Sporadic cases}
There are two more symplectic Grassmannians for which full exceptional
collections have been constructed: on $\IGr(3, 8)$ and $\IGr(3, 10)$
by L.~Guseva and A.~Novikov, respectively.

\begin{theorem}[\cite{Guseva}]\label{thm:igr38}
   There is a Lefschetz decomposition of $\IGr(3, 8)$ with the Lefschetz
   basis given by
   \[
      \langle F, \CO, \CU^*, S^2\CU^*, \Lambda^2\CU^*, \Sigma^{2,1}\CU^* \rangle,
      \qquad
      o = (2, 6, 6, 6, 6, 6),
   \]
   where $F$ is a certain explicit exceptional vector bundle.
   In particular, $D(\IGr(3, 8))$ admits a full exceptional collection.
\end{theorem}

\begin{theorem}[\cite{Novikov}]\label{thm:igr310}
   There is an explicit exceptional vector bundle $H$ on $\IGr(3, 10)$
   and a rectangular Lefschetz semiorthogonal decomposition
   \[
      D(\IGr(3, 10)) = \langle \CB, \CB(1), \ldots, \CB(7) \rangle,
   \]
   where $\CB$ is generated by an exceptional collection
   \[
      \CB = \langle H,\, S^2\CU,\, \Sigma^{3,1}\CU(1),\, \CU,\, \Sigma^{2,1}\CU(1),\,
      \Sigma^{3,1}\CU^*(-1),\, \CO,\, \CU^*,\, S^2\CU^*,\, S^3\CU^* \rangle.
   \]
   In particular, $D(\IGr(3, 10))$ admits a full exceptional collection.
\end{theorem}

\subsection{Odd symplectic Grassmannians}
It is very plausible that another class of varieties which are very close
to being homogeneous admits full exceptional collections. Consider
a vector space $V$ of dimension $2n+1$ equipped with a skew-symmetric form
$\omega$ of maximal rank (which equals $2n$). As in the even case,
one can define isotropic Grassmannians $\IGr(k, V)$ for all $k=1, \ldots, n+1$
as zero loci of the corresponding regular sections of $\Lambda^2\CU^*$
on $\Gr(k, V)$. A good reference for the geometry of these varieties
is~\cite{Mihai2007}. Similarly, one can define odd symplectic flag varieties.

\begin{conjecture}\label{conj:igr-odd}
   Derived categories of odd symplectic flag varieties admit full
   exceptional collections.
\end{conjecture}

When $k=1$, one has $\IGr(k, V)\simeq \PP(V)$. When $k=n+1$, it is
not hard to check that $\IGr(n+1, V)\simeq \LGr(\bar{V})$, where
$\bar{V}=V/\Ker\omega$. Thus, the first nontrivial case
is given by $\IGr(2, V)$.

\begin{theorem}[{\cite[Remark~5.6]{Kuznetsov2008}}]
   There is a rectangular Lefschetz exceptional collection in
   the derived category $\IGr(2, 2n+1)$
   which comes from the collection~\eqref{eq:gr2odd}:
   \[
      D(\IGr(2, 2n+1)) = \begin{pmatrix}
         S^{n-1}\CU^* & S^{n-1}\CU^*(1) & \cdots & S^{n-1}\CU^*(2n-1) \\
         S^{n-2}\CU^* & S^{n-2}\CU^*(1) & \cdots & S^{n-2}\CU^*(2n-1) \\
         \vdots & \vdots &  & \vdots \\
         \CU^* & \CU^*(1) & \cdots & \CU^*(2n-1) \\
         \CO^* & \CO^*(1) & \cdots & \CO^*(2n-1)
      \end{pmatrix}.
   \]
\end{theorem}

The only other two known cases are covered by the following results.

\begin{theorem}[\cite{Fonarev2022a}]
   There is a rectangular Lefschetz exceptional collection in
   the derived category of $\IGr(3, 7)$ given by
   \[
      D(\IGr(3, 7)) = \begin{pmatrix*}[r]
         \Lambda^2\CQ & \Lambda^2\CQ(1) & \Lambda^2\CQ(2) & \Lambda^2\CQ(3) & \Lambda^2\CQ(4) \\
         \CU^* & \CU^*(1) & \CU^*(2) & \CU^*(3) & \CU^*(4) \\
         \CO & \CO(1) & \CO(2) & \CO(3) & \CO(4) \\
         \CU & \CU(1) & \CU(2) & \CU(3) & \CU(4)
      \end{pmatrix*},
   \]
   where $\CQ=V/\CU$ is the universal quotient bundle.
\end{theorem}

\begin{theorem}[\cite{Cattani2023}]
   There is an explicit exceptional vector bundle $\CH$ on $\IGr(3, 9)$
   and a rectangular Lefschetz decomposition
   \[
      D(\IGr(3, 9)) = \langle \CB, \CB(1), \ldots, \CB(6) \rangle,
   \]
   where $\CB$ is generated by an exceptional collection
   \[
      \CB = \langle \CH,\, S^2\CU,\, \CU,\, \Sigma^{2,1}\CU(1),\,
      \Sigma^{3,1}\CU^*(-1),\, \CO,\, \CU^*,\, S^2\CU^* \rangle.
   \]
\end{theorem}

\section{Orthogonal Grassmannians}
\subsection{Orthogonal Grassmannians}
We consider Grassmannians for groups of type $B_n$ and $D_n$.
First, let $\bfG$ be of type $B_n$. We number the simple
roots according to Figure~\ref{fig:bn}.
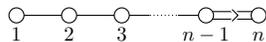
\begin{figure}[H]
   \begin{tikzpicture}
      \dynkin[labels={1,2,3, n-1, n}, scale=2, Bourbaki arrow] B{ooo...oo}
   \end{tikzpicture}
   \caption{Dynkin diagram $B_n$.}
   \label{fig:bn}
\end{figure}
Then $\bfG/\bfP_k$ is isomorphic to the orthogonal Grassmannian 
$\OGr(k, V)$, where $V$ is a $(2n+1)$-dimensional vector space
equipped with a non-degenerate symmetric form $q\in S^2V^*$.
Recall that there is an isomorphism $S^2V^*\simeq \Gamma(\Gr(2, V), S^2\CU^*)$.
Similar to the symplectic case, $\OGr(k, V)$ is the zero locus
of the regular section of $S^2\CU^*$ corresponding to $q$.

Let $\bfG$ be of type $D_n$. We number the simple
roots according to Figure~\ref{fig:dn}.
\begin{figure}[H]
   \begin{tikzpicture}
      \dynkin[labels={1,2,n-2,n-1,n}, label directions={,,right,,}, scale=2, Bourbaki arrow] D{oo...ooo}
   \end{tikzpicture}
   \caption{Dynkin diagram $D_n$.}
   \label{fig:dn}
\end{figure}
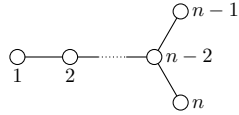

For $k\leq n-2$ the Grassmannian $\bfG/\bfP_k$ is isomorphic to
the orthogonal Grassmannian 
$\OGr(k, V)$, where $V$ is a $2n$-dimensional vector space
equipped with a non-degenerate symmetric form $q\in S^2V^*$.
Again, $\OGr(k, V)$ is the zero locus
of the regular section of $S^2\CU^*$ corresponding to $q$.
When $k=n-1$ or $k=n$, the Grassmannians $\bfG/\bfP_{n-1}$
and $\bfG/\bfP_{n}$ are isomorphic to the two (isomorphic)
connected components of $\OGr(n, V)$, which are traditionally
denoted by $\OGr_+(n, V)$ and $\OGr_-(n, V)$.

Finally, there is a classical isomorphism
of varieties $\OGr_+(n, 2n)\simeq \OGr(n-1, 2n-1)$.

\subsection{Quadrics and orthogonal Grassmannians of planes}
The simplest orthogonal Grassmannian are quadrics.
Consider an $m$-dimensional quadric $Q$. If $m$ is even,
then $Q$ carries two very important vector bundles, called
\emph{spinor bundles}, which we denote by $\CS_+$ and $\CS_-$.
If $m$ is odd, there is only one spinor bundle on $Q$, which
we denote by $\CS$. We refer the reader to~\cite{Ottaviani1988}
for the case of quadrics.
For~spinor bundles on general orthogonal Grassmannians,
we refer the reader to~\cite[Section~6]{Kuznetsov2008}.
It was shown by Kapranov in~\cite{Kapranov1988} that
its bounded derived category admits a full exceptional collection.

\begin{theorem}[\cite{Kapranov1988}]\label{thm:quad}
   Let $Q$ be an $m$-dimensional quadric. There are full exceptional
   collections
      \begin{align}
         \label{eq:qeven}
         D(Q) &=  \langle \CS_+, \CS_-, \CO, \CO(1), \ldots, \CO(m-1) \rangle & & \text{if }m\text{ is even}, \\
         \label{eq:qodd}
         D(Q) &= \langle \CS, \CO, \CO(1), \ldots, \CO(m-1) \rangle & & \text{if }m\text{ is odd}.
      \end{align}
\end{theorem}

\begin{remark}
   Decomposition~\eqref{eq:qodd} is obviously minimal Lefschetz,
   while~\eqref{eq:qeven} can be rewritten as
   \[
      D(Q) =  \langle \CS_+, \CO, \CS_+(1), \CO(1), \ldots, \CO(m-1) \rangle,
   \]
   which is minimal Lefschetz.
\end{remark}

The next two series of cases are $\OGr(2, V)$.
The reader will immediately recognize that the following
collections are closely related to minimal Lefschetz exceptional
collections on classical Grassmannians of planes
(see Theorem~\ref{thm:ld-gr2}).
We first
consider $\IGr(2, 2n+1)$. Full exceptional collections of
vector bundles on these varieties were constructed
by A.~Kuznetsov in~\cite{Kuznetsov2008}.

\begin{theorem}
   Let $n\geq 2$. There is a full Lefschetz exceptional collection on $\IGr(2, 2n+1)$
   of the form
   \begin{equation}\label{eq:ogr2odd}
      D(\OGr(2, 2n+1)) = \begin{pmatrix}
         \CS & \CS(1) & \cdots & \CS(2n-3) \\
         S^{n-2}\CU^* & S^{n-2}\CU^*(1) & \cdots & S^{n-2}\CU^*(2n-3) \\
         \vdots & \vdots &  & \vdots \\
         \CU^* & \CU^*(1) & \cdots & \CU^*(2n-3) \\
         \CO^* & \CO^*(1) & \cdots & \CO^*(2n-3)
      \end{pmatrix},
   \end{equation}
   where $\CS$ denotes the spinor bundle.
\end{theorem}

For $\IGr(2, 2n)$ an almost full exceptional collection was constructed in~\cite{Kuznetsov2008}.
However, it was one object too short. The problem of constructing
a full exceptional collection was solved
by A.~Kuznetsov and M.~Smirnov in~\cite{Kuznetsov2021}.
In its most elegant form it can be stated as follows.

\begin{theorem}[{\cite{Kuznetsov2021}}]
   Consider the subcategory
   \[
      \CA = \langle \CO,\, \CU^*,\, S^2\CU^*,\, \ldots, S^{n-3}\CU^*,\, \CS_-,\,\CS_+\rangle
      \subseteq D(\OGr(2, 2n)),
   \]
   where $\CS_+$ and $\CS_-$ are the two spinor bundles.
   There is a semiorthogonal decomposition
   \[
      D(\OGr(2, 2n)) = \langle \CR,\, \CA,\, \CA(1),\, \ldots,\, \CA(2n-4) \rangle,
   \]
   where $\CR$ is equivalent to the derived category of representations
   of the Dynkin quiver $D_n$. In particular, $\CR$ and, as a consequence,
   $D(\OGr(2, 2n))$ admits a full exceptional collection.
\end{theorem}

\subsection{Kuznetsov--Polishchuk collections}
In~\cite{KuznetsovPolishchuk2016} the authors constructed enough exceptional blocks
for orthogonal Grassmannians to get exceptional collections of maximal
possible length. Accurate formulation of these results
requires some setup from representation theory, so
we refer the reader to~\cite[Theorems~9.1 and~9.3]{KuznetsovPolishchuk2016}
for details.

\begin{theorem}\label{thm:bn}
   Consider $\OGr(k, V)$, where $\dim V = 2n+1$ and $k\leq n-1$.
   For each $0\leq t \leq k-1$ the sets of weights
   \[
      B_t = \left\{\lambda\in \ZZ^n\; \middle|\;
      \begin{aligned}
      & 2n-1-k \geq \lambda_1 \geq \cdots \geq
      \lambda_t \geq t = \lambda_{t+1} = \cdots = \lambda_k, \\
      & (k-t)/2 \geq \lambda_{k+1} \geq \cdots \geq \lambda_n \geq 0
      \end{aligned}
      \right\}
   \] 
   and
   \[
      B_{t+1/2} = \left\{\lambda\in (1/2+\ZZ)^n\; \middle|\;
      \begin{aligned}
      & 2n-1/2-k \geq \lambda_1 \geq \cdots \geq
      \lambda_t \geq t+1/2 = \lambda_{t+1} = \cdots = \lambda_k, \\
      & (k-t)/2 \geq \lambda_{k+1} \geq \cdots \geq \lambda_n \geq 0
      \end{aligned}
      \right\}
   \]
   are (right) exceptional blocks. 
   
   For each $k\leq t \leq 2n-k-1$
   the set of weights
   \[
      B_t = \left\{\lambda\in \ZZ^n\; \middle|\;
      \begin{aligned}
      & 2n-1-k \geq \lambda_1 \geq \cdots \geq
      \lambda_{k-1} \geq \lambda_k=t, \\
      & \lambda_{k+1} = \cdots = \lambda_n = 0
      \end{aligned}
      \right\}
   \]
   is a (right) exceptional block.

   Denote by $\CA_t$ and $\CA_{t+1/2}$
   the categories generated by the exceptional collections
   $\langle \CE^\lambda \mid \lambda \in B_t \rangle$
   and $\langle \CE^\lambda \mid \lambda \in B_{t+1/2} \rangle$,
   respectively. The subcategories
   \[
      \CA_0,\, \CA_{1/2},\, \CA_1,\, \ldots,\, \CA_{k-1/2},\, \CA_k,\, \CA_{k+1},\, \ldots,\, \CA_{2n-1-k}
   \]
   are semiorthogonal.
\end{theorem}

\begin{theorem}\label{thm:dn}
   Consider $\OGr(k, V)$, where $\dim V = 2n$ and $k\leq n-2$.
   For each $0\leq t \leq k-1$ the sets of weights
   \[
      B_t = \left\{\lambda\in \ZZ^n\; \middle|\;
      \begin{aligned}
      & 2n-2-k \geq \lambda_1 \geq \cdots \geq
      \lambda_t \geq t = \lambda_{t+1} = \cdots = \lambda_k, \\
      & (k-t)/2 \geq \lambda_{k+1} \geq \cdots \geq \lambda_{n-1} \geq \abs{\lambda_n}
      \end{aligned}
      \right\}
   \] 
   and
   \[
      B_{t+1/2} = \left\{\lambda\in (1/2+\ZZ)^n\; \middle|\;
      \begin{aligned}
      & 2n-3/2-k \geq \lambda_1 \geq \cdots \geq
      \lambda_t \geq t+1/2 = \lambda_{t+1} = \cdots = \lambda_k, \\
      & (k-t)/2 \geq \lambda_{k+1} \geq \cdots \geq \lambda_{n-1} \geq \abs{\lambda_n}
      \end{aligned}
      \right\}
   \]
   are (right) exceptional blocks. 
   
   For each $k\leq t \leq 2n-k-2$
   the set of weights
   \[
      B_t = \left\{\lambda\in \ZZ^n\; \middle|\;
      \begin{aligned}
      & 2n-2-k \geq \lambda_1 \geq \cdots \geq
      \lambda_{k-1} \geq \lambda_k=t, \\
      & \lambda_{k+1} = \cdots = \lambda_n = 0
      \end{aligned}
      \right\}
   \]
   is a (right) exceptional block.

   Denote by $\CA_t$ and $\CA_{t+1/2}$
   the categories generated by the exceptional collections
   $\langle \CE^\lambda \mid \lambda \in B_t \rangle$
   and $\langle \CE^\lambda \mid \lambda \in B_{t+1/2} \rangle$,
   respectively. The subcategories
   \[
      \CA_0,\, \CA_{1/2},\, \CA_1,\, \ldots,\, \CA_{k-1/2},\, \CA_k,\, \CA_{k+1},\, \ldots,\, \CA_{2n-2-k}
   \]
   are semiorthogonal.
\end{theorem}

\begin{theorem}\label{thm:bnmax}
   Consider $\OGr(n, V)$, where $\dim V = 2n+1$.
   For all $0\leq t \leq n-1$ the sets of weights
   \[
      B_{2t} = \{ \lambda\in\ZZ^n \mid n-1 \geq \lambda_1 \geq \cdots \geq
      \lambda_t \geq t = \lambda_{t+1} = \cdots = \lambda_n \}
   \]
   and
   \[
      B_{2t+1} = \{ \lambda\in(1/2+\ZZ)^n \mid n-1/2 \geq \lambda_1 \geq \cdots \geq
      \lambda_t \geq t+1/2 = \lambda_{t+1} = \cdots = \lambda_n \}
   \]
   are (right) exceptional blocks.

   Denote by $\CA_t$ the subcategory in $D(\OGr(n, V))$ generated by the exceptional collection
   $\langle \CE^\lambda \mid \lambda \in B_t \rangle$.
   The subcategories
   $
      \CA_0,\, \CA_1,\, \ldots,\, \CA_{2n-1}
   $
   are semiorthogonal.
\end{theorem}

\begin{conjecture}
   The collections constructed in Theorems~\ref{thm:bn},
   \ref{thm:dn}, and~\ref{thm:bnmax} are full.
\end{conjecture}

\subsection{Sporadic cases}
We have the following not very interesting cases of orthogonal
Grassmannians when $n$ is small.
\begin{itemize}
   \item If $\bfG$ is of type $B_2$, then $\bfG/\bfP_2\simeq \PP^3$.
   This case is covered by Beilinson's work.
   \item If $\bfG$ is of type $D_3$, then $\bfG/\bfP_1\simeq \bfG/\bfP_3 \simeq \PP^3$.
   This case is covered by Beilinson's work.
   \item If $\bfG$ is of type $D_4$, then $\bfG/\bfP_3\simeq \bfG/\bfP_4 \simeq \bfG/\bfP_1\simeq Q^6$
   This case is covered by Kapranov's work.
   \item If $\bfG$ is of type $B_3$, then
   $\bfG/\bfP_3 = \OGr(3, 7) \simeq \OGr(4, 8) \simeq Q^6$.
   This case is covered by Kapranov's work.
\end{itemize}

The derived category of $\OGr_+(5, 10)$ was first studied by A.~Kuznetsov.
\begin{theorem}[\cite{Kuznetsov2006}]
   Let $X=\OGr_+(5, 10)$, also isomorphic to $\OGr(4, 9)$.
   The bounded derived category of $X$ admits a full exceptional collection.
\end{theorem}

\begin{remark}
   In~\cite{Moschetti2024} R.~Moschetti and M.~Rampazzo show that the collection
   given by the Kuznetsov--Polishchuk construction for $\OGr_+(5, 10)$
   is full.
\end{remark}

The next theorem by V.~Benedetti, D.~Faenzi, and M.~Smirnov
covers $\OGr_+(6, 12)$, which is isomorphic to $\OGr(5, 11)$.
\begin{theorem}[\cite{Benedetti2023}]
   Let $X=\OGr_+(6, 12)$. There are unique $\Spin_{12}$-equivariant
   extensions
   \[
      0\to \CO \to \CE^{(1,1)} \to \Lambda^2\CU^* \to 0
      \quad \text{and} \quad
      0\to \CU^* \to \CE^{(2,1)} \to \Sigma^{2,1}\CU^* \to 0
   \]
   and a Lefschetz decomposition of $D(X)$ with the Lefschetz basis
   \[
      \langle \CO, \CU^*, \CE^{(1,1)}, \CE^{(2,1)} \rangle,
      \quad
      o = (10, 10, 10, 2).
   \]
   In particular, $D(\OGr_+(6, 12))\simeq D(\OGr(5, 11))$
   admits a full exceptional collection of vector bundles.
\end{theorem}

\section{Grassmannians for Exceptional Groups}
\subsection{Grassmannians for $E_6$, $E_7$, $E_8$}
The only case when something is known is when
$\bfG$ is of type $E_6$ and the parabolic
group is either $\bfP_1$ or $\bfP_6$,
where we use the numbering of roots shown on Figure~\ref{fig:e6}.
\begin{figure}[H]
   \begin{tikzpicture}
      \dynkin[labels={1,4,2,3,5,6}, scale=2] E6
   \end{tikzpicture}
   \caption{Dynkin diagram $E_6$.}
   \label{fig:e6}
\end{figure}
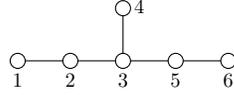

The diagram has a symmetry, so $\bfG/\bfP_1$ and $\bfG/\bfP_6$
are isomorphic. This is a $16$-dimensional variety of Fano index $12$,
which is also known as the Cayley plane $\mathbb{OP}^2$.
Its derived category was originally studied by L.~Manivel in~\cite{Manivel2011}.
The following result is due to D.~Faenzi and L.~Manivel.

\begin{theorem}[\cite{Faenzi2015}]\label{thm:op2}
   There is a Lefschetz decomposition of $D(\mathbb{CO}^2)$
   with the Lefschetz basis given by
   \[
      \langle \CU^{2\omega_6-\omega_1},\, \CU^{\omega_6-\omega_1}, \CO \rangle,
      \quad
      o = (3, 12, 12),
   \]
   where $\CU^{\omega}$ is the vector bundle associated with the $\bfL$-dominant
   weight $\omega$, and $\omega_i$ are the dominant weights of $\bfG$
   with respect to the numbering given on Figure~\ref{fig:e6}.
   In particular, $D(\mathbb{CO}^2)$ admits a full exceptional collection
   consisting on vector bundles.
\end{theorem}

Apart from that when $\bfG$ is of type $E_7$, V.~Benedetti, D.~Faenzi, and M.~Smirnov
in~\cite{Benedetti2023}
construct an exceptional collection of maximal length
in $D(\bfG/\bfP_7)$.

\subsection{Grassmannians for $F_4$}
Let $\bfG$ be of type $F_4$. We use the numbering of weights given
on Figure~\ref{fig:f4}.
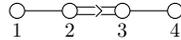
\begin{figure}[H]
   \begin{tikzpicture}
      \dynkin[label, scale=2, Bourbaki arrow] F4
   \end{tikzpicture}
   \caption{Dynkin diagram $F_4$.}
   \label{fig:f4}
\end{figure}

The two cases when something is known correspond to the first and fourth roots.
The variety $\bfG/\bfP_1$ is known as the adjoint Grassmannian in type $F_4$.
It is a variety of dimension $15$, and its Fano index equals $8$.
The following was shown by M.~Smirnov.

\begin{theorem}[\cite{Smirnov2021}]
   Let $X=\bfG/\bfP_1$ be the adjoint Grassmannian.
   There is a unique equivariant extension
   \[
      0 \to \CO \to \tilde{\CT} \to \CT_{X} \to 0
   \]
   of its tangent bundle $\CT_X$ and a rectangular Lefschetz exceptional
   collection
   \[
      D(X) = \begin{pmatrix*}[r]
         \tilde{\CT} & \tilde{\CT}(1) & \tilde{\CT}(2) & \tilde{\CT}(3) & \tilde{\CT}(4) & \tilde{\CT}(5) & \tilde{\CT}(6) & \tilde{\CT}(7) \\
         \CU^{\omega_4} & \CU^{\omega_4}(1) & \CU^{\omega_4}(2) & \CU^{\omega_4}(3) & \CU^{\omega_4}(4) & \CU^{\omega_4}(5) & \CU^{\omega_4}(6) & \CU^{\omega_4}(7) \\
         \CO & \CO (1) & \CO (2) & \CO (3) & \CO (4) & \CO (5) & \CO (6) & \CO (7)
      \end{pmatrix*},
   \]
   where $\CU^{\omega_4}$ is the vector bundle associated with the $\bfL$-dominant
   fundamental weight $\omega_4$ of $\bfG$.
\end{theorem}

The second case is the variety $\bfG/\bfP_4$, which is known
as the coadjoint Grassmannian in type $F_4$.
It is a variety of dimension $15$, and its Fano index equals $11$.
Moreover, it can be realized as a general hyperplane section
of the Cayley plane $\mathbb{CO}^2$. In~\cite{Belmans2023}
P.~Belmans, A.~Kuznetsov, and M.~Smirnov show that the collection constructed by L.~Manivel
and D.~Faenzi (see Theorem~\ref{thm:op2}) restricts to a full
rectangular exceptional collection in $D(\bfG/\bfP_4)$.

\begin{theorem}[{\cite{Belmans2023}}]
   The coadjoint Grassmannian $\bfG/\bfP_4$
   admits a full Lefschetz exceptional collection
   consisting of vector bundles with the support function
   \[
      o=(11, 11, 2).
   \]
\end{theorem}

\subsection{Grassmannians for $G_2$}
Finally, we consider the case when $\bfG$ is of type $G_2$.
The root numbering is shown on Figure~\ref{fig:g2}.
\begin{figure}[H]
   \begin{tikzpicture}
      \dynkin[label, scale=2, Bourbaki arrow] G2
   \end{tikzpicture}
   \caption{Dynkin diagram $G_2$.}
   \label{fig:g2}
\end{figure}
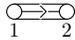

There are only two Grassmannians. The quotient
$\bfG/\bfP_2$ is isomorphic to $Q_5$, a $5$-dimensional quadric.
This case is covered by Kapranov's work. The remaining case
is $\bfG/\bfP_1$, which is a $5$-dimensional variety of Fano index $3$.
A.~Kuznetsov showed the following.

\begin{theorem}[\cite{Kuznetsov2006}]
   The bounded derived category of $\bfG/\bfP_1$ admits a full
   exceptional collection.
\end{theorem}

\printbibliography

\end{document}